\documentclass{amsart}
\pdfoutput=1
\usepackage[margin=3cm]{geometry}
\usepackage{color}
\usepackage{amsmath}
\usepackage{amssymb}
\usepackage{bbm}
\usepackage{tikz-cd}
\usepackage{amsthm}
\usepackage[all,arc,2cell]{xy}
\usepackage{mathrsfs}
\usepackage{mathtools}
\usepackage{hyperref} 
\hypersetup{%
  bookmarksnumbered=true,%
  colorlinks=true,%
  linkcolor=blue,%
  citecolor=blue,%
  filecolor=blue,%
  menucolor=blue,%
  urlcolor=blue,%
  pdfnewwindow=true,%
  pdfstartview=FitBH}   
\usepackage[shortlabels]{enumitem}
\usepackage{import}
\usepackage{graphicx}
\graphicspath{ {images/} }
\usepackage{color}
\usepackage{nicematrix}
\usepackage{transparent}
\usepackage{cleveref}
\creflabelformat{equation}{#2\textup{#1}#3}
\newcommand\restr[2]{{
  \left.\kern-\nulldelimiterspace 
  #1 
  \vphantom{\big|} 
  \right|_{#2} 
  }}

\newtheorem{theorem}{Theorem}[section]
\newtheorem{proposition}[theorem]{Proposition}
\newtheorem{lemma}[theorem]{Lemma}

\theoremstyle{definition}
\newtheorem{example}[theorem]{Example}

\theoremstyle{definition}
\newtheorem{definition}[theorem]{Definition}
\newtheorem*{definition*}{Definition}

\newtheorem{remark}[theorem]{Remark}

\numberwithin{equation}{subsection}

\renewcommand\a{\alpha}
\renewcommand\b{\beta}
\newcommand\g{\gamma}
\newcommand\G{\Gamma}
\renewcommand\d{\delta}
\newcommand\D{\Delta}
\newcommand\e{\epsilon}

\newcommand\z{\zeta}

\newcommand\m{\mu}
\newcommand\n{\nu}
\newcommand\x{\xi}

\renewcommand\r{\rho}

\newcommand\s{\sigma}
\renewcommand\S{\Sigma}
\renewcommand\t{\tau}

\newcommand\ps{\psi}

\renewcommand\c{\chi}

\newcommand{\bbF}{\mathbb{F}}

\newcommand{\bbQ}{\mathbb{Q}}
\newcommand{\bbR}{\mathbb{R}}
\newcommand{\bbZ}{\mathbb{Z}}

\newcommand{\bbT}{\mathbb{T}}

\newcommand{\wta}{\widetilde{\alpha}}
\newcommand{\wtb}{\widetilde{\beta}}
\newcommand{\wtg}{\widetilde{\gamma}}

\newcommand{\ngr}{\operatorname{gr}'}
\newcommand{\gr}{\operatorname{gr}}
\newcommand{\grh}{\widehat{\operatorname{gr}}}

\newcommand{\bsa}{\boldsymbol{\alpha}}
\newcommand{\bsb}{\boldsymbol{\beta}}
\newcommand{\bsg}{\boldsymbol{\gamma}}
\newcommand{\bsd}{\boldsymbol{\delta}}

\newcommand{\bfx}{\mathbf{x}}
\newcommand{\bfy}{\mathbf{y}}
\newcommand{\bfw}{\mathbf{w}}

\newcommand{\sign}{\text{sgn}}
\newcommand{\HF}{\mathit{HF}}
\newcommand{\HFR}{\mathit{HFR}}
\newcommand{\CF}{\mathit{CF}}
\newcommand{\CFR}{\mathit{CFR}}

\newcommand{\cs}{\c_{\mathfrak{s}}}
\newcommand{\ctot}{\c_{\operatorname{tot}}}

\newcommand{\bfxpf}{\mathbf{x}_{\text{first}}^+}
\newcommand{\bfxpl}{\mathbf{x}_{\text{last}}^+}
\newcommand{\bfxmf}{\mathbf{x}_{\text{first}}^-}
\newcommand{\bfxml}{\mathbf{x}_{\text{last}}^-}

\newcommand{\mycomment}[1]{}

\begin{document}

\title{Absolute $\bbZ/2$ Gradings in Real Heegaard Floer Homology}
\author{Eha Srivastava}
\address{Department of Mathematics\\Stanford University\\
		Building 380\\
		Stanford, California 94305}
\email{esrivas@stanford.edu}

\begin{abstract}
Real Heegaard Floer homology is an invariant associated to a three-manifold equipped with an involution with nonempty fixed set of codimension two. We show that when the image of the fixed point set is nullhomologous in the quotient, the real Heegaard Floer homology groups admit an absolute $\bbZ/2$ grading; in particular this applies to double branched covers of links in $S^3$. As an application, we define a $\bbZ$-valued invariant of knots, which is the appropriate signed analogue of Miyazawa's degree invariant. Furthermore, we show that this invariant is equal to the Alexander polynomial of the knot evaluated at $i$. 
\end{abstract}

\maketitle 

\section{Introduction}
In \cite{OS1}, Ozsv\'ath and Szab\'o define a series of Floer theories associated to a closed, oriented three-manifold $Y$, called the Heegaard Floer homology groups and denoted $\mathit{HF}^\circ(Y)$, where $\circ \in \{-, \infty, +, \hat{\text{ }}\}$. 
Recently, Guth and Manolescu defined \emph{real} versions of the Heegaard Floer homology groups in \cite{GM1}, which are associated to a three-manifold $Y$ with an involution $\t$ whose fixed set is nonempty and has codimension two. 

The usual Heegaard Floer theories admit various gradings. There is a natural decomposition 
\[\mathit{HF}^\circ(Y) = \bigoplus_{\mathfrak{s} \in \text{Spin}^c(Y)} \mathit{HF}^\circ(Y, \mathfrak{s})\] 
and each group $\HF^\circ(Y, \mathfrak{s})$ has a relative $\bbZ/\d(\mathfrak{s})$ grading, where 
\[\d(\mathfrak{s}) = \gcd_{\x \in H_2(Y ; \bbZ)} \langle c_1(\mathfrak{s}), \x\rangle.\]
This relative grading can be lifted to an absolute $\bbQ$ grading when the Spin$^c$-structure $\mathfrak{s}$ is torsion; see \cite{OS3}. 
Ozsv\'ath and Szab\'o also show in \cite{OS2} that $\HF^\circ(Y, \mathfrak{s})$ can be endowed with an absolute $\bbZ/2$ grading for any Spin$^c$-structure $\mathfrak{s}$. 
There is an analogue of this $\bbZ/2$ grading in monopole Floer homology; see \cite{KM1}. 
In the Heegaard Floer setting, the grading can be used to pin down signs of the Euler characteristics of these Floer homology theories. This has a number of applications:
\begin{itemize}
    \item It is used in \cite{OS2} to establish a relation between $\chi(\HF^+(Y))$ and Turaev's torsion function. 
    \item In the case that $Y$ is a integral homology three-sphere, this grading is used to prove a relationship between the Euler characteristic of $\HF^+_{\text{red}}(Y)$, the Casson invariant, and the correction term; see Theorem 1.3 of \cite{OS4}. 
    \item The absolute $\bbZ/2$ grading is also a key tool in defining the mixed cobordism invariants. It is used to prove a vanishing result about maps on $\HF^\infty$ induced by certain cobordisms, which is then needed to define the mixed cobordism invariants; see Lemma 8.1 of \cite{OS3}. 
    \item It is used in \cite{AM1} to set up the Heegaard Floer complexes over the integers. 
\end{itemize}

The definition of the absolute $\bbZ/2$ grading in \cite{OS2} is somewhat complicated, as it requires computing the Heegaard Floer homology with totally twisted coefficients and showing that $\underline{\HF}^\infty(Y, \mathfrak{s}) \cong \bbZ[U, U^{-1}]$. 
Declaring this to be supported in even degree determines the grading. Working in the more general context of sutured Floer homology, Friedel, Juh\'asz, and Rasmussen proposed an alternative definition of a $\bbZ/2$ grading in \cite{FJR1}. 
In constrast to the original definition, their definition only uses the cohomological data of an associated Heegaard diagram. 
The grading in \cite{FJR1} introduces an additional factor of $\pm 1$ depending on the rank of the first homology group of the manifold, and differs from the grading in \cite{OS2} by this factor. 
This grading was specialized to the case of closed three-manifolds by Petkova in \cite{P1}. 
Petkova does not include this additional factor, so that the grading in \cite{P1} agrees with the the original grading in \cite{OS2}. This is explicitly shown for three-manifolds whose first homology has rank zero.
As a warm-up to the real case, we study the grading in \cite{P1}, presenting it in a slightly different manner. 
For a three-manifold $Y$ and an associated Heegaard diagram $(\S, \bsa, \bsb, z)$, we first construct a function $\ngr : \CF^\infty(\bbT_\a, \bbT_\b) \to \bbZ/2$ using the Heegaard diagram. We then fill in the details of the proof that this grading is well-defined and agrees with the one in \cite{OS2} for all three-manifolds. 

\begin{theorem}
    Let $Y$ be a closed oriented three-manifold, and let $(\S, \bsa, \bsb, z)$ be an admissible pointed Heegaard diagram for $Y$. The function $\ngr : \mathit{CF}^\infty(\bbT_\a, \bbT_\b) \to \bbZ/2$ determines a well-defined absolute $\bbZ/2$ grading $\ngr$ on $\HF^\infty(Y)$. Moreover, this grading agrees with $\text{\emph{gr}}$, the absolute $\bbZ/2$ grading defined in \cite{OS2}.  
    \label{main-theorem1}
\end{theorem}

Similarly, the real Heegaard Floer homology groups $\HFR^\circ(Y, \t)$ admit a decomposition by real Spin$^c$-structures on $(Y, \t)$, and each $\HFR^\circ(Y, \t, \mathfrak{s})$ admits various gradings. 
When $\delta(\mathfrak{s}) = 2N$, the corresponding group $\HFR^\circ(Y, \t, \mathfrak{s})$ admits a relative $\bbZ/N$ grading. 
It is a bit harder to define an absolute $\bbZ/2$ grading. The chain complexes $\CF(L_1, L_2)$ in Lagrangian Floer homology can usually be given an absolute $\bbZ/2$ grading by orienting the Lagrangians $L_1$ and $L_2$. 
This cannot always be done in the real setting, as the Lagrangians used to define the real Heegaard Floer complexes are not always orientable.
However, we show that under certain assumptions, the hat version of the real Heegaard Floer groups admits an absolute $\bbZ/2$ grading:

\begin{theorem}
   Let $L$ be an oriented nullhomologous link in a closed oriented three-manifold $X$, and fix an ordering of the link components. 
   Fix a primitive class $S \in H_2(X, L ; \bbZ)$ such that $\partial S = [L]$. 
   Let $Y$ be the double branched cover of $X$ along $L$ determined by the image of $S$ in $H_2(X, L ;\bbZ/2)$, and let $\t$ denote the branching involution. 
   Fix a collection of basepoints $\bfw$ such that there is one basepoint on each component of the fixed point set. 
   Then the hat version of the real Heegaard Floer homology $\widehat{\HFR}(Y, \t, \bfw)$ of $(Y, \t, \bfw)$ admits an absolute $\bbZ/2$ grading $\grh$, depending on the ordering and orientation of $L$, and the choice of primitive class $S$. 
   \label{main-theorem2}
\end{theorem}

\begin{remark} 
Any real three-manifold $(Y, \t)$ is the double branched cover of the quotient $X = Y /\t$ along the projection $L \subset X$ of the fixed set. 
The link $L$ is nullhomologous in $H_2(X, L ;\bbZ/2)$.
\Cref{main-theorem2} then says that when this link $L$ is nullhomologous in $H_2(X, L ; \bbZ)$, the hat version of the real Heegaard Floer homology of the real three-manifold $(Y, \t)$ admits an absolute $\bbZ/2$ grading. 
\end{remark}

The grading in \Cref{main-theorem2} is constructed using a pointed real Heegaard diagram $(\S, \bsa, \bsb, \mathbf{w}, \t)$ for $(Y, \t)$ such that the quotient $\S/\t$ is a connected Seifert surface of $L$. 
In this case, removing a basepoint from each component of the fixed point set $C$ in $Y$ allows both Lagrangians to be oriented, and the ordering and orientations of the link components specifies natural orientations. 
Any two such Heegaard diagrams are connected by certain real Heegaard moves which preserve this grading. 

As described in Section 7 of \cite{GM1}, there are two ways to consider the Euler characteristic of $\widehat{\HFR}(Y, \t, \bfw,\mathfrak{s})$. In the real case, the decomposition by real Spin$^c$-structures admits a refinement into relative real Spin$^c$-structures. 
If a pair of intersection points $\bfx$ and $\bfy$ determine the same real Spin$^c$-structure $\mathfrak{s}$, where $\d(\mathfrak{s})$ is divisible by four, and $\bfx$ and $\bfy$ belong to the same relative real Spin$^c$-structure, then the real Maslov index determines a relative $\bbZ/2$ grading between these two points. 
This relative $\bbZ/2$ grading allows us to compute the Euler characteristic $\c(\widehat{\HFR}(Y, \t, \mathbf{w}, \mathfrak{s}))$ up to sign. 
Alternatively, when the Lagrangians can be oriented, comparing signs of intersection points gives rise to another relative $\bbZ/2$ grading. 
We let $\widehat{\c}(\widehat{\HFR}(Y, \t, \mathbf{w}, \mathfrak{s}))$ denote the Euler characteristic with respect to this grading. 
Since this relative grading is determined across all real Spin$^c$-structures, we can also compute $\widehat{\c}(\widehat{\HFR}(Y, \t, \bfw))$. 
Again, this is only defined up to sign. As Example 7.4 of \cite{GM1} illustrates, these Euler characteristics $\chi$ and $\widehat{\chi}$ need not agree. 

Since the grading in \Cref{main-theorem2} is constructed by fixing orientations of the Lagrangians, we will only be considering the second Euler characteristic $\widehat{\c}$. 
In this setting, the absolute $\bbZ/2$ grading allows us to pin down the sign of the Euler characteristic  $\widehat{\c}(\widehat{\HFR}(Y, \t, \bfw))$. 
We may also fix the sign of the hat version of the Euler characteristic in each real Spin$^c$-structure. 
These signs turn out to be independent of the ordering of the link components.

\begin{theorem}
    Let $L \subset X$ be an oriented nullhomologous link in a closed oriented three-manifold. 
    Fix a primitive class $S \in H_2(X, L ; \bbZ)$ such that $\partial S = [L]$, and let $(Y, \t)$ be the double branched cover of $X$ along $L$ determined by $S$ with branching involution $\t$. 
    Fix a collection of basepoints $\bfw$ such that there is one basepoint on each component of the fixed point set.
    Then for each real Spin$^c$-structure $\mathfrak{s}$, the sign of the hat version of the Euler characteristic $\widehat{\c}(\widehat{\HFR}(Y, \t, \bfw,\mathfrak{s}))$ depends only on the oriented link $L$ and the primitive homology class $S$. 
    It is also unchanged under simultaneously replacing $L$ with its reverse $rL$ and $S$ with $-S$. 
    \label{main-theorem3}
\end{theorem}

We then specialize further to the case of double branched covers $\S_2(L)$ of links $L$ in $S^3$, equipped with the branching involution $\t_L$. 
Each Spin$^c$-structure on $\S_2(L)$ has a unique real Spin$^c$-structure, so the decomposition by real $\text{Spin}^c$-structures can simply be written as 
\[\widehat{\HFR}(\S_2(L), \t_L, \bfw) = \bigoplus_{\mathfrak{s} \in \text{Spin}^c(\S_2(L))}\widehat{\HFR}(\S_2(L), \t_L, \bfw, \mathfrak{s}).\]
For each $\mathfrak{s} \in \text{Spin}^c(\S_2(L))$, let $\cs(L)$ denote the hat version of the Euler characteristic in that $\text{Spin}^c$-structure.
In \cite{GM1}, Guth and Manolescu define the Heegaard Floer analogue of Miyazawa's degree invariant for knots $K$ (see \cite{Miy1}) as 
\[\ctot(K) = \lvert \widehat{\c}(\widehat{\HFR}(\S_2(K), \t_K, \bfw))\rvert.\]
Because $\cs(K)$ is a priori only well-defined up to sign, Guth and Manolescu fix the signs globally so that 
\[\ctot(K) = \sum_{\mathfrak{s} \in \text{Spin}^c(\S_2(K))}\cs(K).\]
Using the absolute $\bbZ/2$ grading $\grh$, we define $\bbZ$-valued invariants of links $L$, which are the appropriate signed analogues of Miyazawa's invariant. 
As a remark, our definitions of $\ctot$ and $\cs$ differ from the one in \cite{GM1} by a sign.

\begin{definition}
    Let $L \subset S^3$ be a link, and let $(\S_2(L), \t_L)$ be the double branched cover of $S^3$ over $L$ along with the branching involution. Define 
    \[\ctot(L) = \widehat{\c}(\widehat{\HFR}(\S_2(L), \t_L, \bfw )),\]
    and for each $\mathfrak{s \in }\text{Spin}^c(\S_2(L))$, define  
    \[\cs(L) = \widehat{\c}(\widehat{\HFR}(\S_2(L), \t_L, \bfw, \mathfrak{s})),\]
    where $\widehat{\HFR}(\S_2(L), \t_L)$ is given the grading $\grh$ from \Cref{main-theorem2}.
    \label[definition]{deg-def}
\end{definition} 

We show that $\ctot$ can actually be expressed in terms of the Alexander polynomial:

\begin{proposition}
    Let $L \subset S^3$ be an $l$-component link in $S^3$. Then 
    \[\ctot(L) = 2^{l-1}\D_{L}(i, \ldots, i).\] 
    Here, $\D_L$ denotes the \mycomment{symmetrized Alexander polynomial, and $A_L$ denotes the}multivariate Alexander polynomial. In particular, for a knot $K \subset S^3$, we have $\ctot(K) = \D_K(i).$
    \label[proposition]{alex-poly}
\end{proposition}

The structure of this paper is as follows: In Section 2, we present the alternate definition of the $\bbZ/2$ grading for usual Heegaard Floer homology and prove \Cref{main-theorem1}. 
Section 3 introduces the absolute $\bbZ/$2 grading for $\widehat{\HFR}$ and shows that it is well-defined. 
Finally, Section 4 looks at the Euler characteristic of $\widehat{\HFR}$. 

\subsection*{Acknowledgements} 
I would like to thank Ciprian Manolescu for many helpful conversations and continued support.
I am also grateful to Ciprian Bonciocat, Gary Guth, and Judson Kuhrman for useful discussions, and to Mohammed Abouzaid for the formulation of the grading in Section 2. 
Finally, I am thankful for the support provided by the Simons Collaboration Grant on New Structures in Low-Dimensional Topology. 

\section{The absolute \texorpdfstring{$\bbZ/2$}{Z/2} grading for Heegaard Floer homology}

\subsection{The definition of the grading}
We start by presenting a definition of a $\bbZ/2$ grading on the Heegaard Floer homology groups of a closed oriented three-manifold $Y$ in terms of the cohomological data of an associated Heegaard diagram. 
As a remark, this grading is the same as the one defined in \cite{P1}, although it is presented in a slightly different manner. 
Furthermore, this grading differs from the one in \cite{FJR1} by a factor of $(-1)^{b_1(Y)}$. Let $(\S, \boldsymbol{\a}, \boldsymbol{\b}, z)$ be a pointed genus $g$ Heegaard diagram for $Y$, and assume that the corresponding tori $\mathbb{T}_\a$ and $\mathbb{T}_\b$ intersect transversely in $\text{Sym}^g(\S)$. 
Let $U_\a$ and $U_\b$ denote the handlebodies specified by the alpha and beta curves, respectively. 
Consider the Mayer-Vietoris sequence in cohomology with coefficients in $\bbR$ for $(Y, U_\a, U_\b, \S)$: 
\[0 \longrightarrow H^1(Y ) \xrightarrow[]{(i_\a^*, i_\b^*)} H^1(U_\a) \oplus H^1(U_\b ) \xrightarrow[]{j_\a^* - j_\b^*} H^1(\S ) \xrightarrow[]{\hspace{.15cm}\delta\hspace{.15cm}} H^2(Y ) \longrightarrow 0.
\]
This breaks up into the following short exact sequences:
\begin{equation} 0  \longrightarrow \text{Im}(j_\a^* - j_\b^*) \longrightarrow H^1(\S ) \longrightarrow H^2(Y ) \longrightarrow 0
\label{ses-1}
\end{equation}
and
\begin{equation} 0 \longrightarrow H^1(Y ) \longrightarrow H^1(U_\a ) \oplus H^1(U_\b ) \longrightarrow \text{Im}(j_\a^* - j_\b^*) \longrightarrow 0.
\label{ses-2}
\end{equation}

Given such a short exact sequence of vector spaces, an orientation on any two of the spaces determines an orientation on the third. 
Our convention will be to fix an isomorphism identifying the middle vector space with the direct sum of the left and right vector spaces, in that order, such that this isomorphism also identifies the short exact sequence with the trivial one. 
We then require our vector spaces to be oriented so that this isomorphism is orientation preserving. Order and orient the alpha and beta curves arbitrarily. 
Then $H^1(U_\a )$ is generated by the algebraic duals $\{\wta_i^*\}$ of $g$ simple closed curves $\{\wta_i\}$ satisfying $\a_i \cdot \wta_j = \delta_{ij}$, and similarly for $H^1(U_\b )$.  

We orient $H^1(\S)$ by declaring that the symplectic basis $\langle \a_1^*, \wta_1^*, \ldots, \a_g^*, \wta_g^*\rangle$ with respect to the symplectic form given by the intersection pairing is positively oriented. 
To orient $H^1(Y )$ and $H^2(Y )$, fix an arbitrary orientation on $H^1(Y )$. 
Poincar\'e duality and the universal coefficients theorem yield isomorphisms $H^1(Y ) \cong H_2(Y ) \cong H^2(Y )$, and we let the image of a positively oriented basis of $H^1(Y )$ under these isomorphisms be a positively oriented basis of $H^2(Y )$. 
This induces an orientation of $\text{Im}(j_\a^* - j_\b^*)$ using (\ref{ses-1}), which in turn induces an orientation of $H^1(U_\a ) \oplus H^1(U_\b )$ using (\ref{ses-2}).

\begin{lemma}
    The orientation of $H^1(U_\a ) \oplus H^1(U_\b )$ is independent of the ordering and orientations of the alpha and beta curves, and of the orientation on $H^1(Y )$.
\end{lemma}

\begin{proof}
    Fix arbitrary orderings and orientations of the alpha and beta curves. Denote these $\{\a_1, \ldots, \a_g\}$ and $ \{\b_1, \ldots, \b_g\}$. 
    Without loss of generality, change the orientation on $\a_1$. Replacing $\a_1$ with $-\a_1$ replaces the basis $\langle \a_1^*, \wta_1^*, \ldots, \a_g^*, \wta_g^* \rangle$ of $H^1(\S )$ with $\langle -\a_1^*, -\wta_1^*, \ldots, \a_g^*, \wta_g^* \rangle$.
    These clearly determine the same orientation of $H^1(\S)$. Similarly, let $\r$ be a  permutation of $\{1,\ldots,g\}$ and consider another ordering $\{\a_{\r(1)}, \ldots, \a_{\r(g)}\}$ of the alpha curves. 
    Then the change of basis matrix between the corresponding bases of $H^1(\S )$ has determinant $\text{sgn}(\r)^2 = 1$. 
    Changing the ordering and orientations of the beta curves does not affect anything either.

    Changing the orientation on $H^1(Y )$ changes the orientation on $H^2(Y )$, and thus also the orientation on $\text{Im}(j_\a^* - j_\b^*)$ obtained using (\ref{ses-1}). 
    So, we obtain the same orientation on $H^1(U_\a ) \oplus H^1(U_\b )$ using (\ref{ses-2}). 
\end{proof}

There is a canonical isomorphism $H^1(U_\a ) \oplus H^1(U_\b ) \to T_\mathbf{x}\bbT_\a \oplus T_\mathbf{x}\bbT_\b $ for each $\mathbf{x} \in \bbT_\a \cap \bbT_\b$. 
We can describe this map after fixing orientations and orderings of the alpha and beta curves: Write $\bfx = \{x_1, \ldots, x_g\}$, where $x_i \in \a_{i} \cap \b_{\s(i)}$. 
Let $v_i \in T_{x_i}\a_i$ and $w_i \in T_{x_i}\b_{\s(i)}$ denote positive tangent vectors with respect to the orientations of the alpha and beta curves. 
Then the generators $\wta_i^*$ and $\wtb_i^*$ get mapped to the vectors $v_i$ and $w_i$, respectively. 
Using this isomorphism, we obtain an orientation of $T_\mathbf{x}\bbT_\a \oplus T_\mathbf{x}\bbT_\b$ for each $\mathbf{x} \in T_\mathbf{x} \bbT_\a \cap T_\mathbf{x} \bbT_\b$. 

\begin{definition}
For each $\mathbf{x} \in T_\mathbf{x} \bbT_\a \cap T_\mathbf{x} \bbT_\b$, define a sign $\sign(\mathbf{x}) \in \{\pm 1\}$ which is $1$ whenever the orientation on $T_\mathbf{x}\text{Sym}^g(\S)$ induced by the orientation on $\S$ agrees with the one induced by the orientation on $T_\mathbf{x}\mathbb{T}_\a \oplus T_\mathbf{x}\mathbb{T}_\b$ constructed above. Define a function $\ngr : \bbT_\a \cap \bbT_\b \to \bbZ/2$ by  
\[\ngr(\mathbf{x}) = 0\]
if and only if $(-1)^g \cdot \sign(\mathbf{x}) = 1$.
\label[definition]{new_gr}
\end{definition}

\begin{remark}
    The $(-1)^g$ factor ensures that the grading remains invariant under stabilizations. 
    The need for such a correction arises from the fact that the map $j^*_\a - j^*_\b: H^1(U_\a) \oplus H^1(U_\b ) \to H^1(\S)$ introduces a minus sign when mapping each $\wtb_{i}^*$ to its restriction in $H^1(\S)$.  
\end{remark}

This determines a $\bbZ/2$-valued function on the chain complexes $\CF^\circ(\bbT_\a, \bbT_\b)$, and this function turns out to be a lift of the usual relative $\bbZ/2$ grading defined using the Maslov index. 
To see this, we note the following properties of the Maslov index and local intersection numbers:

\begin{lemma} 
    Let $(\S, \bsa, \bsb, \bsg, z)$ be a pointed Heegaard triple. Fix orderings and orientations of the alpha, beta, and gamma curves, and a point $\boldsymbol{\Theta} \in \bbT_\b \cap \bbT_\g$. 
    For an intersection point $\bfx$, let $\e(\bfx)$ denote the local intersection number with respect to the obvious product orientations on $\bbT_\a, \bbT_\b$, and $\bbT_\g$. 
    Suppose $\e(\boldsymbol{\Theta}) = 1$. For points $\bfx, \bfy \in \bbT_\a \cap \bbT_\b$ and a holomorphic strip $\phi \in \pi_2(\bfx, \bfy)$, 
    \[(-1)^{\m(\phi)} = \e(\bfx) \e(\bfy).\]
    In particular, a chain map defined by counting strips with Maslov index zero preserves $\ngr$. 
    Similarly, for points $\bfx \in \bbT_\a\cap \bbT_\b$ and $\bfy \in \bbT_\a \cap \bbT_\g$, and a holomorphic triangle $\phi \in \pi_2(\bfx, \boldsymbol{\Theta}, \bfy)$, 
    \[(-1)^{\m(\phi)} = \e(\bfx)\e(\boldsymbol{\Theta})\e(\bfy).\]
    In particular, a triangle map determined by counting triangles $\phi \in \pi_2(\bfx, \boldsymbol{\Theta}, \bfy)$ of Maslov index one preserves $\ngr$ if and only if the orientations on $T_\bfx\bbT_\a \oplus T_\bfx \bbT_\b$ and $T_\bfy\bbT_\a \oplus T_\bfy \bbT_\g$ constructed for \Cref{new_gr} either both agree or both disagree with the obvious product orientations. 
\label[lemma]{maslov} 
\end{lemma}

\begin{proof}
    This follows from standard properties of the Maslov index in Lagrangian Floer homology; see Sections 1.3 and 2.1 of \cite{A1}.
\end{proof}

Passing to homology, we obtain an absolute $\bbZ/2$ grading on $\HF^\circ(\bbT_\a, \bbT_\b)$, and it remains to verify that this grading is independent of the choice of Heegaard diagram. 
For this, we need the following result (Lemma 2.8 in \cite{FJR1}) on the local intersection numbers of $\bbT_\a$ and $\bbT_\b$ equipped with the standard product orientations induced by orderings and orientations of the alpha and beta curves:

\begin{lemma}
    Let $(\S, \a_1, \ldots, \a_g, \b_1, \ldots, \b_g)$ be a Heegaard diagram with orderings and orientations of the alpha and beta curves. 
    Fix the corresponding product orientations on $\bbT_\a$ and $\bbT_\b$. 
    Let $\bfx = \{x_1, \ldots, x_g\} \in \bbT_\a \cap \bbT_\b$, and let $\s$ be the permutation such that $x_i \in \a_i \cap \b_{\s(i)}$. 
    Let also $\e(x_i)$ denote the local intersection number of $\a_i$ and $\b_{\s(i)}$ at $x_i$. 
    Then the sign of $\bfx$ as an intersection point between $\bbT_\a$ and $\bbT_\b$ is given by   
     \[\e(\bfx) = (-1)^{\frac{g(g-1)}{2}}\text{\emph{sgn}}(\s)\prod_{i = 1}^g\e(x_i).\]
     \label[lemma]{int-nums}
\end{lemma}

With this, we can verify that $\ngr$ is a well-defined grading on $\HFR^\circ(Y)$.

\begin{proposition}
    Let $Y$ be a closed oriented 3-manifold, and let $(\S, \bsa, \bsb, z)$ be a pointed admissible Heegaard diagram for $Y$. 
    The absolute $\bbZ/2$ grading on $\HF^\circ(\bbT_\a, \bbT_\b)$ defined above is independent of the choice of Heegaard diagram. 
    \label[proposition]{gr-inv}
\end{proposition}

\begin{proof}
    Fix orderings and orientations of the alpha and beta curves.
    Without loss of generality, assume $\langle \wta_1^* \ldots, \wta_g^*, \wtb_1^*, \ldots, \wtb_g^* \rangle$ is a positively oriented basis of $H^1(U_\a ) \oplus H^1(U_\b )$ according to the construction of \Cref{new_gr}. 
    Then for each $\mathbf{x} \in \bbT_\a \cap \bbT_\b$, the orientation on $T_\mathbf{x} \bbT_\a \oplus T_\mathbf{x} \bbT_\b$ agrees with the product orientation induced by the orderings and orientations of the alpha and beta curves. In other words, $\sign(\bfx) = \e(\bfx)$. 

    We first verify invariance under isotopies. 
    Note that isotopies of the attaching curves don't change the corresponding classes in homology. 
    The isomorphism between the corresponding homology groups constructed in the proof of Theorem 7.3 in \cite{OS1} is obtained from a chain map counting disks with Maslov index zero. 
    By \Cref{maslov}, this must preserve $\ngr$. 

    Next, we verify invariance under stabilizations. 
    Let $(\S^+, \bsa^+, \bsb^+, z)$ denote the stabilization of the original Heegaard diagram. 
    Label the new alpha and beta curves $\a_{g+1}$ and $\b_{g+1}$. 
    Orient $\a_{g+1}$ and $\b_{g+1}$ so that they intersect positively at a single point $x_{g+1}$. 
    There is a clear correspondence between the sets of intersection points: a point $\bfx \in \bbT_\a \cap \bbT_\b$ corresponds to the point $\bfx^+ = \bfx \times \{x_{g+1}\} \in \bbT_\a^+ \cap \bbT_\b^+$. 
    We claim that for each $\bfx \in \bbT_\a \cap \bbT_\b$, 
    \[\ngr(\bfx) = \ngr(\bfx^+).\]
    We first compute the orientation on $H^1(U_\a^+ ) \oplus H^1(U_\b^+)$. 
    Fix an orientation of $H^1(Y)$. 
    For brevity, let $\psi = j_\a^* - j_\b^*$ and $\psi^+ = j_{\a^+}^* - j_{\b^+}^*$. 
    Let $B$ be a positively oriented basis of $\text{Im}(\psi)$. 
    Then there is an orientation-preserving isomorphism 
    \[f : \text{Im}(\psi) \oplus H^2(Y ) \longrightarrow H^1(\S ).\]
    The stabilized diagram $(\S^+, \bsa^+, \bsb^+, z)$ is obtained by taking the connect sum with a standard genus one Heegaard diagram for $S^3$, so that $\psi^+(\wtb_{g+1}^*) = \a_{g+1}^*$, $\psi^+(\wta_{g+1}^*) = \wta_{g+1}^*$, and 
    \[\text{Im}(\psi^+) = \text{Im}(\psi) \oplus \langle \a_{g+1}^*, \wta_{g+1}^*\rangle.\]
    Define an isomorphism $f^+ : \text{Im}(\psi^+) \oplus H^2(Y ) \to H^1(\S^+ )$ by letting $f^+$ be the inclusion on $\text{Im}(\psi^+)$ and letting $f^+$ be equal to $f$ on $H^2(Y )$. 
    Since $f$ is orientation-preserving, $B \oplus \langle \a_{g+1}^*, \wta_{g+1}^*\rangle$ is a positively oriented basis of $\text{Im}(\psi^+)$. 

    Similarly, there is an isomorphism $h : H^1(Y ) \oplus \text{Im}(\psi)  \to H^1(U_\a ) \oplus H^1(U_\b )$. 
    Let $M$ denote the change of basis matrix between the image of this isomorphism and the basis $\langle \wta_1^*, \ldots, \wta_g^*, \wtb_1^*, \ldots, \wtb_g^*\rangle$, so that $\det(M) > 0$.  
    Again, use $h$ to construct an isomorphism $h^+ : H^1(Y ) \oplus \text{Im}(\psi)  \oplus\langle\a_{g+1}^*, \wta^*_{g+1}\rangle \to H^1(U_\a^+ ) \oplus H^1(U_\b^+ )$. 
    Then the change of basis matrix $M^+$ between the image of $h^+$ and $\langle \wta_1^*, \ldots, \wta_{g+1}^*, \wtb_1^*, \ldots, \wtb_{g+1}^*\rangle$ has determinant $(-1)^{g+1}\cdot \det(M)$. 
    Therefore, $\sign(\bfx^+) = (-1)^{g+1}\e(\bfx^+)$.  

    By \Cref{int-nums}, $\e(\bfx^+) = (-1)^g\e(\bfx)$. Combining this with the computations above, we see 
    \begin{align*}
        (-1)^{g+1}\sign(\bfx^+) &= (-1)^{g+1}(-1)^{g+1}\e(\bfx^+) \\
        &= \e(\bfx^+)\\
        &= (-1)^g\e(\bfx) \\
        &= (-1)^g\sign(\bfx).
    \end{align*}
    Therefore, $\ngr(\bfx^+) = \ngr(\bfx)$.

    Finally, we check invariance under handleslides. 
    Without loss of generality, assume we handleslide $\a_1$ over $\a_2$ to obtain a new curve $\g_1$. 
    Orient this new curve so that $\g_1 = \a_1 - \a_2$. 
    Note that $\g_1$ intersects $\wta_1$ positively, since $\g_1 \cdot \wta_1 = (\a_1 - \a_2) \cdot \wta_1 = 1$, so that we may let $\wtg_1 = \wta_1$.
    For $1 < i \leq g$, let $\g_i$ denote a curve isotopic to $\a_i$, intersecting $\a_i$ transversely in two points. 
    It is easy to see that the bases $\langle \a_1^*, \wta_1^*, \ldots, \a_g^*, \wta^*_g\rangle$ and $\langle \g_1^*, \wtg_1^*, \ldots, \g_g^*, \wtg_g^*\rangle$ determine the same orientation on $H^1(\S)$. 
    It follows that $\langle \wtg_1^*, \ldots, \wtg_g^*, \wtb_1^*, \ldots, \wtb_g^*\rangle$ is a positively oriented basis for $H^1(U_\g ) \oplus H^1(U_\b )$, so that for any $\mathbf{y} \in \bbT_\g \cap \bbT_\b$, the orientation of $T_\bfy\bbT_\g \oplus T_\bfy\bbT_\b$ used in \Cref{new_gr} agrees with the one induced by the orderings and orientations of the gamma and beta curves. 
    Invariance now follows from the construction of the maps proving invariance of Heegaard Floer homology under handleslides, which are induced by certain triangle maps (Theorem 9.5 in \cite{OS1}). 
    In particular, by \Cref{maslov}, we have an isomorphism between $\HF^\circ(\bbT_\g, \bbT_\b)$ and $\HF^\circ(\bbT_\a, \bbT_\b)$ which preserves $\ngr$. 
    The case of sliding one of the beta curves over another is similar. 
\end{proof}

\subsection{Relation to gr}
Having shown that the grading $\ngr$ constructed in the previous section is a well-defined grading on the Heegaard Floer homology groups, we turn to showing the second part of \Cref{main-theorem1}. 
Namely, we show that this grading agrees with the absolute $\bbZ/2$ grading $\gr$ defined by Ozsv\'ath and Szab\'o in \cite{OS2}.

\begin{proof}[Proof of \Cref{main-theorem1}]
We induct on $b_1(Y)$. The case where $Y$ is a closed oriented three-manifold with $b_1(Y) = 0$ is Corollary 3.4 in \cite{P1}. 

When $b_1(Y) = n > 0$, let $K \subset Y$ be a knot representing a primitive element of $H_1(Y  ; \bbZ) /\text{Tors}$. 
Remove a tubular neighborhood $N(K)$ of $K$ and attach a solid torus by identifying a meridian to a longitude on $\partial(Y - N(K))$. 
This results in another closed, oriented three-manifold $X$ such that $b_1(X) = b_1(Y) - 1$ and the knot $K$ is nullhomologous in $X$. 
Let $X_p$ denote the three-manifold obtained by integral surgery on $K \subset X$ with framing $p$. 
Then $Y = X_0$ and $X_1$ is a three-manifold with $b_1(X_1) = b_1(X)$. 
There is a long exact sequence (see Theorem 9.1 in \cite{JM1}) relating the fully twisted Heegaard Floer homologies of $X, X_0$, and $X_1$: 
\[\dots \longrightarrow \underline{\HF}^+(X)[T, T^{-1}] \xrightarrow[]{\underline{F_1^+}}\underline{\HF}^+(X_0) \xrightarrow[]{\underline{F_2^+}} \underline{\HF}^+(X_1)[T, T^{-1}] \xrightarrow[]{\underline{F_3^+}} \dots\]

Both $\underline{F_1^+}$ and $\underline{F_3^+}$ preserve $\gr$ while $\underline{F_2^+}$ shifts it. 
These three maps are obtained from triangle maps associated to a suitably admissible pointed Heegaard multi-diagram $(\S, \bsa, \bsb, \bsg, \bsd, z)$ for the triple $(X, X_0, X_1)$ satisfying the following properties:
\begin{itemize}
    \item The Heegaard diagrams $(\S, \boldsymbol{\a},  \boldsymbol{\b})$, $(\S, \boldsymbol{\a},  \boldsymbol{\g})$, and $(\S, \boldsymbol{\a}, \boldsymbol{\d})$ describe $X, X_0$, and $X_1$ respectively. 
    \item For $i = 1, \ldots, g-1$, the curves $\g_i$ and $\delta_i$ are small isotopic translates of $\b_i$, each pairwise intersecting in a pair of cancelling intersection points. 
    \item The curve $\d_g$ is isotopic to the juxtaposition of $\b_g$ and $\g_g$ with appropriate signs.
\end{itemize}
By stabilizing if necessary, assume that $(-1)^g = (-1)^{g(g-1)/2} = (-1)^{g(g+1)/2} = 1$. 
Following the proof of Theorem 9.1 in \cite{OS2}, orient $\a_1, \ldots, \a_g$ and $\b_1, \ldots, \b_{g-1}$ arbitrarily. 
This induces orientations on $\g_1, \ldots, \g_{g-1}$ and $\d_1, \ldots, \d_{g-1}$, following the isotopies of the beta curves. 
The orientations on $\b_g$ and $\d_g$ are determined by requiring that the absolute $\bbZ/2$ gradings on $\CF^\infty(\bbT_\a, \bbT_\b)$ and $\CF^\infty(\bbT_\a, \bbT_\d)$ induced by the orientations of the alpha, beta, and delta curves agree with Ozsv\'ath and Szab\'o's grading, which in turn agrees with the new grading by the inductive hypothesis. 
Orient $\g_g$ so that $\d_g = \b_g - \g_g$.

For $i < g$, label the intersection points $y_i^\pm = \b_i \cap \g_i, v_i^\pm = \g_i \cap \d_i,$ and $w_i^\pm = \b_i \cap \d_i$, where the sign indicates the sign of the intersection at that point. 
Also, let $y_g = \b_g \cap \g_g, v_g = \g_g \cap \d_g,$ and $w_g = \b_g \cap \d_g$. 
Then let 
\[\boldsymbol{\Theta}_{\b, \g} = \{y_1^+, \ldots, y_{g-1}^+, y_g\}, \hspace{.2cm} \boldsymbol{\Theta}_{\g, \d} = \{v_1^+, \ldots, v_{g-1}^+, v_g\}, \hspace{.2cm} \boldsymbol{\Theta}_{\b, \d} = \{w_1^+, \ldots, w_{g-1}^+, w_g\}\]
be the corresponding intersection points in $\bbT_\b \cap \bbT_g, \bbT_\g \cap \bbT_\d, $and $\bbT_\b \cap \bbT_\d$.
Note that $\b_g \cdot \g_g = 1$ and $\g_g \cdot \d_g = -1$, so that by \Cref{int-nums}$, \e(\boldsymbol{\Theta}_{\b, \g}) = 1$ and $\e(\boldsymbol{\Theta}_{\g, \d}) = -1$.
On the chain level, $\underline{F_1^+}$ and $\underline{F_2^+}$ are induced by the maps 
\[\underline{f_1}^+([\mathbf{x}, i]) = \sum_{\mathbf{y} \in \mathbb{T}_\a \cap \mathbb{T}_\g} \sum_{\{\phi \in \pi_2(\mathbf{x}, \boldsymbol{\Theta}_{\b, \g}, \mathbf{y}) \hspace{.03cm} \vert \hspace{.03cm}\mu(\phi) = 0\}} c(\mathbf{x}, \mathbf{y}, \phi) \cdot [\mathbf{y}, i - n_z(\phi)]\]
and 
\[\underline{f_2}^+([\mathbf{x}, i]) = \sum_{\mathbf{y} \in \mathbb{T}_\a \cap \mathbb{T}_\d} \sum_{\{\phi \in \pi_2(\mathbf{x}, \boldsymbol{\Theta}_{\g, \d}, \mathbf{y}) \hspace{.03cm} \vert \hspace{.03cm}\mu(\phi) = 0\}} c(\mathbf{x}, \mathbf{y}, \phi) \cdot [\mathbf{y}, i - n_z(\phi)],\]
where $c(\mathbf{x}, \mathbf{y}, \phi) \in \bbZ[T, T^{-1}]$ depends on $\mathbf{x}, \mathbf{y}$, and $\phi$.
By \Cref{maslov}, both $\underline{F_1^+}$ and $\underline{F_3^+}$ will preserve $\ngr$ while $\underline{F_2^+}$ will shift it if and only if $\sign(\bfx) = \e(\bfx)$ for every $\bfx \in \bbT_\a \cap \bbT_\g$. 
Since $\g_g \cdot \b_g = -1$ and $\b_i \simeq \g_i$ for $i < g$, it suffices to show that $\langle \wta_1^*, \ldots, \wta_g^*, \wtb_1^*, \ldots, \wtb_{g-1}^*, -\b_g^*\rangle$ is a positively oriented basis of $H^1(U_\a ) \oplus H^1(U_\g )$ according to the construction of Section 2.1. 

Fix an orientation on $H^1(X )$, inducing an orientation on $H^2(X )$. 
Orient $H^1(X_0 )$ as $H^1(X_0 ) = \langle \b_g^* \rangle \oplus H^1(X )$, so that $H^2(X_0 )$ is also oriented as $H^2(X_0 ) = \langle PD^{-1}(\b_g)\rangle \oplus H^2(X )$. 
Again, write $\psi = j_\a^* - j_\b^*$ and $\psi_0 =j_\a^* - j_\g^*$. 
We have that $\text{Im}(\psi) = \text{Im}(\psi_0) \oplus \langle \psi(\wtb_g^*)\rangle$, so choose a basis 
$B$ of $\text{Im}(\psi_0)$ such that $B \oplus \langle \psi(\wtb_g^*)\rangle$ is a positively oriented basis of $\text{Im}(\psi).$
We claim that the basis $B$ is a positively oriented basis of $\text{Im}(\psi_0)$. 
For, we have an orientation-preserving isomorphism $f : \text{Im}(\psi) \oplus H^2(Y ) \to H^1(\S )$. 
Define an isomorphism $f_0 : \text{Im}(\psi_0) \oplus H^2(X_0 ) \to H^1(\S )$ by 
\[\restr{f_0}{\text{Im}(\psi_0)} = \restr{f}{\text{Im}(\psi_0)}, \hspace{.25cm} \restr{f_0}{H^2(X )} = \restr{f}{H^2(X )}, \hspace{.25cm} f_0(PD^{-1}(\b_g)) = f(\psi(\wtb_g^*)).\]

For this to satisfy the appropriate commutativity conditions, we need that $\delta(\psi(\wtb_g^*)) = PD^{-1}(\b_g)$, or equivalently that $PD(\delta(\psi(\wtb_g^*))) = \b_g$. 
Since the image of $\psi(\wtb_g^*)$ under the boundary map in the Mayer-Vietoris sequence for $(X, U_\a, U_\b, \S)$ is trivial, we need only check that $\b_g^*(PD(\delta(\psi(\wtb_g^*)))) = \b_g^*(\b_g) = 1$.
Computing this, 
\begin{align*}
    \b_g^*(PD(\d(\psi(\wtb_g^*)))) &= \b_g^*([X_0] \frown \delta(\psi(\wtb_g^*))) \\
    &= (\d(\psi(\wtb_g^*)) \smile\b_g^*)[X_0] \\
    &= PD(\d(\ps(\wtb_g^*))) \cdot PD(\b_g^*) \\
    &= (-\b_g)\cdot(S \text{ s.t. }\partial S = -\g_g) = 1.
\end{align*}

Similarly, we orient $H^1(U_\a ) \oplus H^1(U_\g )$. 
Let $C$ denote a positively oriented basis of $H^1(X)$. 
By induction, there is an isomorphism $h : H^1(X ) \oplus \text{Im}(\psi) \to H^1(U_\a ) \oplus H^1(U_\b )$ such that the change of basis matrix $M$ between the image of the basis $C\oplus B \oplus \langle \psi(\wtb_g^*)\rangle$ and $\langle \wta_1^*, \ldots, \wta_g^*, \wtb_1^*, \ldots, \wtb_g^*\rangle$ has positive determinant. 
Recalling that $H^1(X_0 ) = \langle \b_g^*\rangle \oplus H^1(X )$, define an isomorphism $h_0 : H^1(X_0 ) \oplus \text{Im}(\psi_0) \to H^1(U_\a ) \oplus H^1(U_\g )$ by 
\[\restr{h_0}{\text{Im}(\psi_0)} = \restr{h}{\text{Im}(\psi_0)}, \hspace{.25cm} \restr{h_0}{H^1(X )} = \restr{h}{H^1(X )},\hspace{.25cm} h_0(\b_g^*) = i_\a^*(\b_g^*) + i_\b^*(\b_g^*) = \sum_{i = 1}^g(\a_i \cdot \b_g)\wta_i^* + \b_g^*.\]
Then the change of basis matrix between the image of the basis $\langle \b_g^*\rangle \oplus  C\oplus B$ under $h_0$ and the basis $\langle \wta_1^*, \ldots, \wta_g^*, \wtb_1^*, \ldots, \wtb_{g-1}^*, -\b_g^*\rangle$ has determinant 
\[ \det
\begin{bNiceArray}{c|ccc}[margin, columns-width=auto]
\a_1 \cdot \b_g & \Block[B]{3-3}{\text{Im}\left(\restr{h}{C \oplus B}\right)} & & \\
\Vdots & & &\\
\a_g \cdot \b_g & & &\\
0 & & &\\
\Vdots & & &\\
0 & & & \\
\hline
-1 & 0 & \Cdots &0
\end{bNiceArray}
= (-1)^{2g-1}(-1) \cdot\det(M) > 0.
\]

Therefore, for any $\bfx \in \bbT_\a \cap \bbT_\g$, $\sign(\bfx) = \e(\bfx)$. 
It follows that $\underline{F_1}^+$ and $\underline{F_3}^+$  preserve $\ngr$ and $\underline{F_2}^+$ reverses it. 
By induction, $\ngr$ must agree with Ozsv\'ath and Szab\'o's absolute $\bbZ/2$ grading for $X_0 = Y$ as well. 
\end{proof}

\section{The absolute \texorpdfstring{$\bbZ/2$}{Z/2} grading for real Heegaard Floer homology}
\subsection{Background on real Heegaard Floer homology} 
We first recall some basic facts about real Heegaard Floer homology. 
See \cite{GM1} for full details. 

A \emph{pointed real three-manifold} $(Y, \t, \mathbf{w})$ is a closed oriented three-manifold $Y$, an orientation-preserving involution $\t$ on $Y$ whose fixed point set $C$ is nonempty and has codimension two, and a set of basepoints $\mathbf{w} \in Y$. 
A \emph{real Heegaard splitting} for $(Y, \t)$ is a decomposition $Y = U \cup \t(U)$ into two handlebodies, where one handlebody is the image under $\t$ of the other. 
A \emph{multi-pointed real Heegaard diagram} for $(Y, \t, \mathbf{w})$ is the data $(\S, \a_1, \ldots, \a_m, \b_1, \ldots, \b_m, \mathbf{w}, \t)$.
Here, $m = g(\S) + |\mathbf{w}| -1$ and we require that $\mathbf{w}$ is a collection of basepoints contained in the fixed set $C$. 
We will only consider the case where $\bfw$ has exactly one basepoint on each component of $C$.
Both $\bsa = \{\a_i\}$ and $\bsb = \{\b_i\}$ are collections of disjoint, simple closed curves in $\S$ which bound compressing disks in $U$ and $\t(U)$ respectively, and such that each component of $\S \setminus \bsa$ and $\S \setminus \bsb$ contains exactly one basepoint. 
Finally, we require that $\t : \S \to \S$ exchanges $\bsa$ and $\bsb$. 
As in the regular case, any two real Heegaard diagrams representing a given real three-manifold are related by a sequence of real Heegaard moves. 

In this case, the real Heegaard Floer homology groups are defined by first fixing a real Heegaard diagram $(\S, \bsa, \bsb, \mathbf{w}, \t)$ realizing $(Y, \t, \mathbf{w})$ and considering the symmetric product $M = \text{Sym}^m(\S)$, which inherits an involution $R$ from $\t$. 
This admits a symplectic form with respect to which the torus $\bbT_\a = \a_1 \times \ldots \times \a_m$ and the fixed point set $M^R$ are Lagrangian. 
In its simplest form, the real Heegaard Floer homology of $(Y, \t, \bfw)$ is the Lagrangian Floer homology of $\bbT_\a$ and $M^R$ in $\text{Sym}^m(\S)$.

\subsection{The definition of the grading}
Let $L \subset X$ be an oriented nullhomologous $l$-component link in a closed orientable three-manifold $X$. 
Recall that double branched covers of $X$ along $L$ are classified by classes $S \in H_2(X, L ; \bbZ/2)$ satisfying $\partial S = [L]$.
Fix a primitive homology class $S \in H_2(X, L ; \bbZ)$ satisfying $\partial S = [L]$. 
Let $(Y, \t)$ be the double branched cover of $X$ along $L$ determined by the image of $S$ under the map $H_2(X, L ; \bbZ) \to H_2(X, L ; \bbZ/2)$, where $\t$ is the branching involution. 
Let $C$ denote the fixed point set of $\t$ in $Y$. 
Fix an ordering of the link components, and let $\mathbf{w} = \{w_1, \ldots, w_l\}$ be a set of basepoints such that $w_i$ lies on component $C_i$. 
To construct the grading, we use a certain real Heegaard diagram for $(Y, \t, \bfw)$. 
Because $L$ is nullhomologous in $X$ and $S$ is a primitive element of $H_2(X, L ; \bbZ)$, the real three-manifold $(Y, \t)$ admits a real Heegaard splitting $Y = U \cup_{\S}\t(U)$ such that the quotient $\S' = \S/\t$ is connected and orientable. 

We say that a Seifert surface for a link is \emph{free} if the complement of a neighborhood of the surface is a handlebody. 
As in Proposition 3.6 in \cite{GM1}, we work backwards using a free Seifert surface $\S' \subset X$ satisfying $[\S'] = S$ to construct a real Heegaard diagram for $(Y, \t, \bfw)$. 
Take a neighborhood $\n(\S') \cong \S' \times [0, 1] \subset X$ of $\S'$ whose complement is a handlebody. 
There is a natural involution $\t$ on the boundary $\partial \n(\S') \cong (\S' \times \partial I) \cup (L \times I)$ which swaps the two components of $\S' \times \partial I$ and reflects $L \times I$ through $L \times \{1/2\}$. 
Take representatives for a basis of homology for $H_1(\S'; \bbZ)$ which bound compressing disks in $X \setminus \n(\S')$; these determine alpha curves $\a_1, \ldots, \a_g$ on $\partial \nu( \S') =\S$. 
Choose $l-1$ more alpha curves on $\S$ as follows: around each of the basepoints $w_2, \ldots, w_l$, take a small simple closed curve encircling that basepoint, disjoint from all the other alpha curves. 
The involution determines the beta curves, and we get a Heegaard diagram $(\S, \bsa, \bsb, \bfw, \t)$ for $(Y, \t, \mathbf{w})$. 
\begin{figure}[htbp]
    \centering
    \def\svgwidth{.4\textwidth}
\begingroup%
  \makeatletter%
  \providecommand\color[2][]{%
    \errmessage{(Inkscape) Color is used for the text in Inkscape, but the package 'color.sty' is not loaded}%
    \renewcommand\color[2][]{}%
  }%
  \providecommand\transparent[1]{%
    \errmessage{(Inkscape) Transparency is used (non-zero) for the text in Inkscape, but the package 'transparent.sty' is not loaded}%
    \renewcommand\transparent[1]{}%
  }%
  \providecommand\rotatebox[2]{#2}%
  \newcommand*\fsize{\dimexpr\f@size pt\relax}%
  \newcommand*\lineheight[1]{\fontsize{\fsize}{#1\fsize}\selectfont}%
  \ifx\svgwidth\undefined%
    \setlength{\unitlength}{810bp}%
    \ifx\svgscale\undefined%
      \relax%
    \else%
      \setlength{\unitlength}{\unitlength * \real{\svgscale}}%
    \fi%
  \else%
    \setlength{\unitlength}{\svgwidth}%
  \fi%
  \global\let\svgwidth\undefined%
  \global\let\svgscale\undefined%
  \makeatother%
  \begin{picture}(1,0.69444444)%
    \lineheight{1}%
    \setlength\tabcolsep{0pt}%
    \put(0,0){\includegraphics[width=\unitlength,page=1]{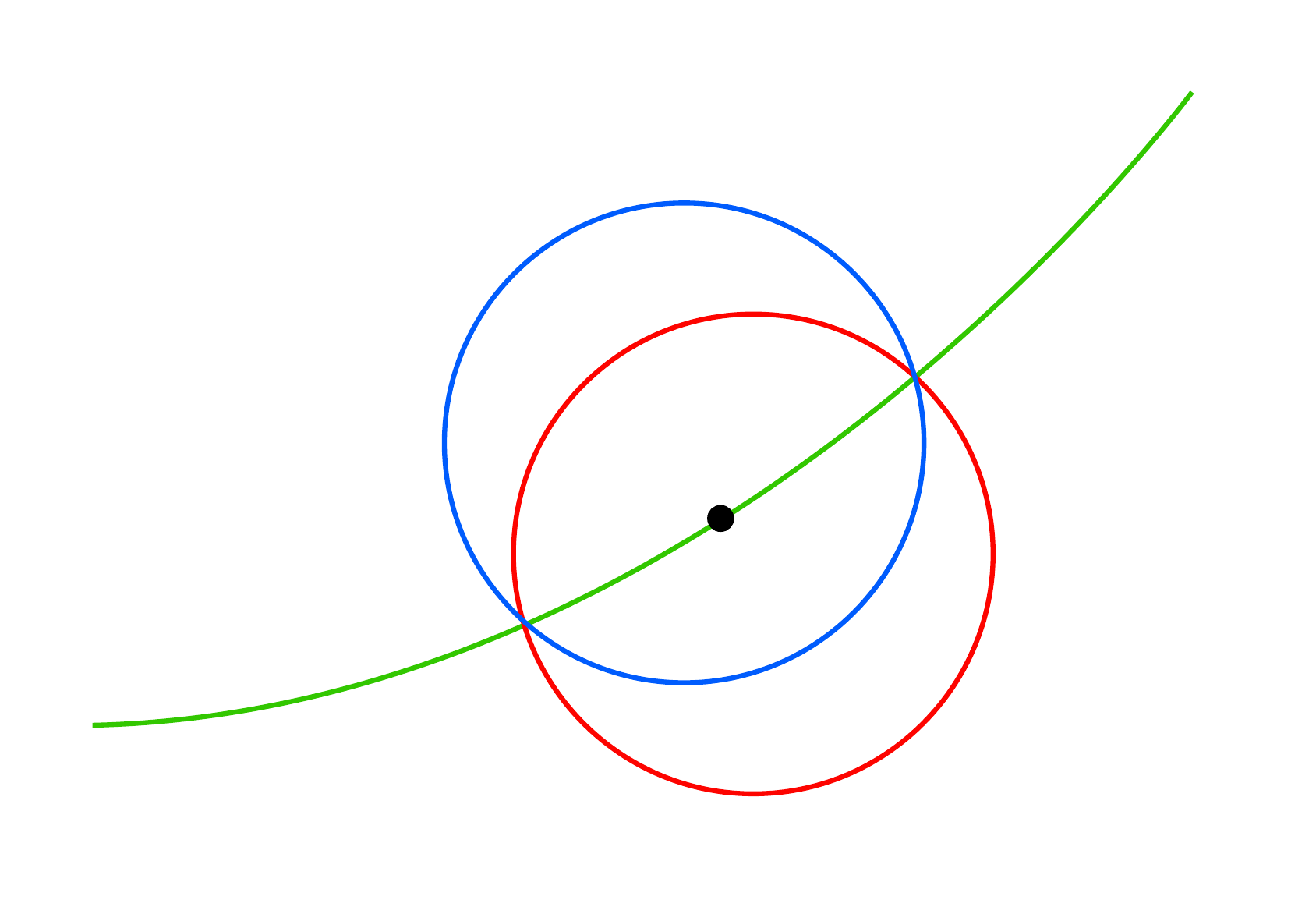}}%
    \put(0.57000162,0.2664284){\color[rgb]{0,0,0}\transparent{0.99536997}\makebox(0,0)[lt]{\smash{\begin{tabular}[t]{l}$w_i$\end{tabular}}}}%
  \end{picture}%
\endgroup%

    \caption{An example of a pair of additional alpha and beta curves.}
    \label{fig:special-alpha}
\end{figure}

Let $m = g + l-1$ denote the number of alpha curves, and let $M = \text{Sym}^m(\S)$ denote the $m$-fold symmetric product. Let $\widehat{M} = M \setminus (\mathbf{w} \times \text{Sym}^{m-1}(\S))$. 
Since there is a basepoint on each component of $C$, the resulting Lagrangian $\widehat{M^R} = M^R \cap \widehat{M}$ becomes orientable.
Fixing an absolute $\bbZ/2$ grading can then be accomplished by fixing orientations of the Lagrangians $\bbT_\a$ and $\widehat{M^R}$ in $M$.
Following Section 7 of \cite{GM1}, $\widehat{M^R}$ breaks up into strata of the form 
\[\text{Sym}^k(\S') \times \text{Sym}^{m - 2k}(\widehat{C})\]
where $\widehat{C} = C \setminus \mathbf{w}$. Orient $\S'$ so that the orientation induced on its boundary agrees with the given one on $L$. 
The orientations and ordering of the link components also determine orientations of the symmetric products $\text{Sym}^{m - 2k}(\widehat{C})$. 
We refer to Section 7.1 of \cite{GM1} for more details.

To orient the other Lagrangian, observe that choosing an orientation of $\bbT_\a$ is equivalent to choosing an orientation of $H_1(U_\a, \mathbf{w} ; \bbR)$. 
Consider the long exact sequence in homology of the triple $(X, X - \S', \mathbf{w})$:
\[0 \longrightarrow H_2(X,\mathbf{w}) \longrightarrow H_2(X, X- \S') \longrightarrow H_1(X - \S', \mathbf{w}) \longrightarrow H_1(X, \mathbf{w}) \longrightarrow 0.\]
Here, we are taking homology with coefficients in $\bbR$. 
The long exact sequence of the pair $(X, \mathbf{w})$ yields isomorphisms $H_2(X, \mathbf{w}) \cong H_2(X)$ and $H_1(X, \mathbf{w}) \cong H_1(X) \oplus \ker(H_0(\mathbf{w} ) \to H_0(X))$. 
So, the exact sequence above can be rewritten as 
\begin{equation}
0 \longrightarrow H_2(X) \longrightarrow H_2(X, X- \S') \longrightarrow H_1(X - \S', \mathbf{w}) \longrightarrow H_1(X) \oplus \ker(H_0(\mathbf{w} ) \to H_0(X))\longrightarrow 0.
\label{triple-les}
\end{equation}
Orient $H_1(X)$ arbitrarily and use Poincar\'e duality to orient $H_2(X)$. 
The ordering of the link components determines an orientation of $\ker(H_0(\mathbf{w} ) \to H_0(X))$. 
Because $X$ is closed and orientable, the pair $(X, \S')$ satisfies the following form of duality (see Theorem 3.44 in \cite{H1}):
\[H_2(X, X - \S') \cong H^1(\S').\]
Let $\widehat{\S'}$ denote the closed oriented surface given by capping $\S'$ off with a connected planar domain. 
Using the long exact sequence of the pair $(\widehat{\S'}, \S')$, we have 
\begin{align*}
    H_1(\widehat{\S'}) &\cong H_1(\S') \oplus H_1(\widehat{\S'}, \S') \\
    &\cong H_1(\S') \oplus \ker(H_0(\mathbf{w} ) \to H_0(X)).
\end{align*}
The isomorphism $H_1(\widehat{\S'}, \S') \to \ker(H_0(\mathbf{w} ) \to H_0(X))$ is given by the boundary homomorphism. 
Because $\widehat{\S'}$ is oriented, $H_1(\widehat{\S'})$ is a symplectic vector space with respect to the intersection form, and we declare a symplectic basis to be positively oriented. 
The ordering of the link components again determines an orientation of $\ker(H_0(\mathbf{w} ) \to H_0(X))$. 
Requiring the isomorphism above to be orientation-preserving then orients $H_1(\S')$. 
Therefore, every vector space other than $H_1(X - \S', \mathbf{w}) = H_1(U_\a, \mathbf{w})$ in the exact sequence of (\ref{triple-les}) is oriented. 
As in Section 2, (\ref{triple-les}) breaks up into the short exact sequences 
\[0 \longrightarrow \text{Im}(H_2(X, X-\S') \to H_1(X-\S', \bfw)) \longrightarrow H_1(X - \S', \bfw) \longrightarrow H_1(X) \oplus\ker(H_0(\bfw)\to H_0(X)) \longrightarrow 0\]
and
\[0 \longrightarrow H_2(X) \longrightarrow H_2(X, X-\S') \longrightarrow \text{Im}(H_2(X, X-\S') \to H_1(X-\S', \bfw)) \longrightarrow 0.\]
Again, we use the convention that if two of the vector spaces appearing in a short exact sequence are oriented, the orientation on the third is determined by fixing an isomorphism identifying the middle vector space with the direct sum of the left and the right, such that the isomorphism identifies this short exact sequence with the trivial one. 
Using these two short exact sequences, we obtain an orientation on $H_1(X - \S', \bfw ) = H_1(U_\a, \mathbf{w})$. 
As seen before, the resulting orientation is independent of the choice of orientation of $H_1(X)$. 
Similarly, it is also independent of the ordering of the link components. 

For an element $\mathbf{x} \in \bbT_\a \cap \widehat{M^R}$, let $\sign(\bfx) \in \{\pm1\}$ denote the resulting sign of $\bfx$ as an intersection point between $\bbT_\a$ and $\widehat{M^R}$ in $\text{Sym}^m(\S)$. 
The grading $\grh(\bfx) : \bbT_\a \cap \widehat{M^R} \to \bbZ/2$ is then defined as
\[\grh(\bfx) = 0\]
if and only if $\sign(\bfx) = 1$. 

In Section 7 of \cite{GM1}, it is described how to compute the grading of an intersection point in terms of intersection numbers of the alpha curves with the fixed set $C$ in $\S$. 
Let $\mathbf{x} \in \bbT_\a \cap \widehat{M^R}$ be a generator. 
Since $\t$ is orientation-reversing on $\S$, pulling back the orientation on $\S'$ to $\S$ splits $\S$ into two disjoint pieces, $\S^+$ and $ \S^-$, according to whether the orientation pulled back from $\S'$ agrees or disagrees with the orientation on $\S$. 
Using the decomposition of $\widehat{M^R}$ into strata, $\mathbf{x}$ is of the form
\[\mathbf{x} = \{z_1, \t(z_1), \ldots, z_k, \t(z_k), c_1, \ldots, c_{m - 2k}\} \in \text{Sym}^{k}(\S') \times \text{Sym}^{m - 2k}(\widehat{C}),\]
where
\[z_i \in \a_{r(i)} \cap \b_{\s(r(i))} \cap \S^+,  \hspace{.2cm}i = 1, \ldots, k\]
and 
\[c_j \in \a_{s(j)} \cap \b_{s(j)}, \hspace{.2cm}j = 1, \ldots, m -2k.\]
We also require that the points $c_1, \ldots, c_{m - 2k}$ appear on the components of $C$ in order, following the ordering and orientation of the link. 
Let $\e(z_i) \in \{\pm 1\}$ denote the sign of intersection between $\a_{r(i)}$ and $\b_{\s(r(i))}$ in $\S^+$, and let $\e(c_j)$ denote the sign of intersection between $\a_{s(j)}$ and $C$ using the orientation on $\S$. 
Finally, let $\e(r, \s, s)$ be the sign of the following permutation of $\{1, \ldots, m\}$:
\[(r(1), \s(r(2)), \ldots, r(k), \s(r(k)), s(1), \ldots, s(m - 2k)).\]
Proposition 7.3 in \cite{GM1} shows that the sign of $\mathbf{x}$ as an intersection point between $\bbT_\a$ and $\widehat{M^R}$ in $\text{Sym}^m(\S)$ is given by 
\begin{equation}
    \sign(\bfx) = \e(r, \s, s) \cdot \prod_{i = 1}^k\e(z_i) \cdot \prod_{j = 1}^{m - 2k}\e(c_j).
    \label{int-sign}
\end{equation}

To show that the grading $\grh$ is well-defined, namely that it depends only on the data of $(Y, \t)$, an ordering and orientation of the link components, and the primitive homology class $S \in H_2(X, L ; \bbZ)$, we need the following propositions. 

\begin{proposition}
    Let $L \subset X$ be an oriented nullhomologous link in a three-manifold $X$. 
    Suppose that $\S_1$ and $\S_2$ are two free Seifert surfaces for $L$ such that $[\S_1] = [\S_2]$ in $H_2(X, L ; \bbZ)$. 
    Then $\S_1$ and $\S_2$ are related by a sequence of stabilizations and destabilizations such that each intermediate Seifert surface is also free. 
    \label[proposition]{seifert-surfaces}
\end{proposition}

\begin{proof}
First, we claim that because $[\S_1] = [\S_2]$ in $H_2(X, L;\bbZ)$, they are related by a sequence of stabilizations and destabilizations. 
We follow the proof of the result in the case of Seifert surfaces of oriented knots in $S^3$ given in Section 9 of \cite{Gor1}.
Consider the embedded surface $M = (\S_1 \times 0) \cup (L \times I) \cup (-\S_2 \times 1) \subset X \times I$. 
Since $[\S_1] = [\S_2]$, $M$ is nullhomologous in $X \times I$, so $M = \partial W$ for some three-manifold $W \subset X \times I$.
In fact, we may take $W$ to be embedded; see page 50 of \cite{Kir1}.
Fix a handle decomposition of $W$ on $\S_1 \times 0$ with only 1- and 2-handles. 
Then the Seifert surface given by intersecting $W$ with $X \times \{t\}$ at a level $t$ in between the $1$- and $2$-handles is obtained from both $\S_1$ and $\S_2$ by adding or removing tubes.

Thus, there is a sequence of Seifert surfaces $\S_1 = F_0, F_1, \ldots, F_n = \S_2$ such that each $F_i$ is obtained from $F_{i-1}$ by attaching or removing a tube $H_i$. 
However, the intermediate Seifert surfaces need not be free. 
To remedy this, we proceed as follows. 
In Lemma 3.5 of \cite{GM1}, the authors describe a procedure for turning an arbitrary Seifert surface $F$ into a free Seifert surface $F_T$ in the same relative homology class. 
We recount the procedure here. Start by fixing a decomposition of $F$ into disks and bands. 
Let $\G \subset F$ be the embedded graph consisting of the cores of the disks and bands of $F$. 
Take a triangulation $T$ of $X$ such that the core graph $\G$ consists of the union of vertices and edges of $T$. 
The surface is then modified using $T$. 
For each vertex not in $F$, an additional disk is added, and for each edge not in $F$, do a band sum. 
For the new bands created, add a dual band to ensure that the boundary of the resulting surface $F_T$ is still isotopic to the original link. 
Note that this procedure amounts to adding tubes along all the edges of $T$ not already contained in $F$.

After an isotopy, we can assume that the tube $H_i$ that is attached (or removed) from $F_{i-1}$ to get to $F_i$ is small and unknotted. 
More precisely, there exists a small ball containing $H_i$, inside of which $H_i$ is unknotted; see Figure \ref{fig:add-tube} below.
For each $i > 0$, let $T_i$ be a triangulation of $X$ satisfying the following two properties:
\begin{enumerate}[(a)]
    \item The core graph $\G_i$ of $F_i$ is a subcomplex of $T_i$.
    \item The triangulation $T_i$ restricts to a triangulation of the tube $H_i$. 
\end{enumerate}

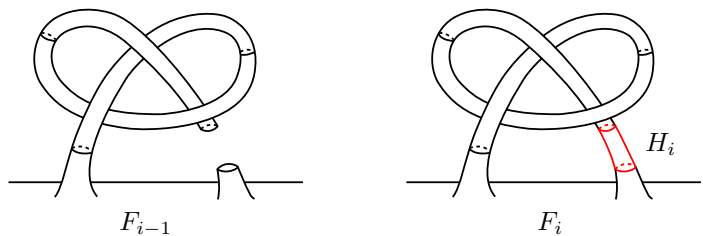
\begin{figure}[htbp]
    \centering
    \def\svgwidth{.62\textwidth}
\begingroup%
  \makeatletter%
  \providecommand\color[2][]{%
    \errmessage{(Inkscape) Color is used for the text in Inkscape, but the package 'color.sty' is not loaded}%
    \renewcommand\color[2][]{}%
  }%
  \providecommand\transparent[1]{%
    \errmessage{(Inkscape) Transparency is used (non-zero) for the text in Inkscape, but the package 'transparent.sty' is not loaded}%
    \renewcommand\transparent[1]{}%
  }%
  \providecommand\rotatebox[2]{#2}%
  \newcommand*\fsize{\dimexpr\f@size pt\relax}%
  \newcommand*\lineheight[1]{\fontsize{\fsize}{#1\fsize}\selectfont}%
  \ifx\svgwidth\undefined%
    \setlength{\unitlength}{535.01847875bp}%
    \ifx\svgscale\undefined%
      \relax%
    \else%
      \setlength{\unitlength}{\unitlength * \real{\svgscale}}%
    \fi%
  \else%
    \setlength{\unitlength}{\svgwidth}%
  \fi%
  \global\let\svgwidth\undefined%
  \global\let\svgscale\undefined%
  \makeatother%
  \begin{picture}(1,0.35272899)%
    \lineheight{1}%
    \setlength\tabcolsep{0pt}%
    \put(0,0){\includegraphics[width=\unitlength,page=1]{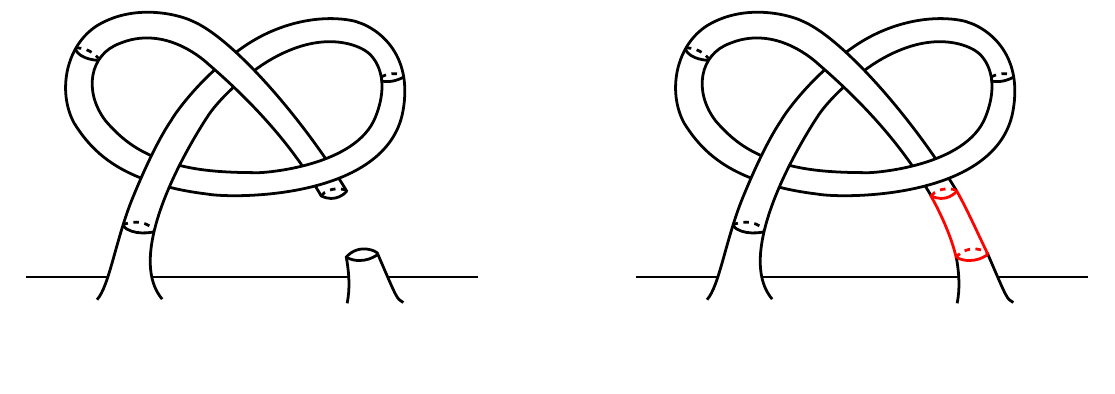}}%
    \put(0.17592694,0.03808601){\color[rgb]{0,0,0}\makebox(0,0)[lt]{\smash{\begin{tabular}[t]{l}$F_{i-1}$\end{tabular}}}}%
    \put(0.75113053,0.03808601){\color[rgb]{0,0,0}\makebox(0,0)[lt]{\smash{\begin{tabular}[t]{l}$F_i$\end{tabular}}}}%
    \put(0.89960301,0.14685305){\color[rgb]{0,0,0}\makebox(0,0)[lt]{\smash{\begin{tabular}[t]{l}$H_i$\end{tabular}}}}%
  \end{picture}%
\endgroup%

    \caption{Adding a small unknotted tube to $F_{i-1}$.}
    \label{fig:add-tube}
\end{figure}

Let $F_{i, T_j}$ be the surface obtained by performing the procedure from \cite{GM1} described above to $F_i$ using the triangulation $T_j$. 
Define a \emph{local stabilization} of a Seifert surface $F$ to be the connect sum of $F$ with $T^2$, where $T^2$ is an unknotted torus inside a ball. 
Note that freeness is preserved by local stabilizations.
We claim that there is a sequence of free Seifert surfaces 
\begin{equation}
    \S_1 = F_0, F_{0, T}, F_{0, T_1}, F_{1, T_1}, F_{2, T_1}, F_{2, T_2}, \ldots, F_{n, T_n}, F_{n, T'}, F_n = \S_2
    \label{eqn:seifert-surface-seq}
\end{equation}
such that consecutive surfaces are related by some sequence of local stabilizations and destabilizations. 
Here, $T$ and $T'$ will be some triangulations of $X$ satisfying (a) for $\S_1$ and $\S_2$, respectively. Showing the existence of a sequence as in (\ref{eqn:seifert-surface-seq}) amounts to proving the following three results:

\begin{enumerate}[(1)]
    \item For a free Seifert surface $F$, there exists a triangulation $T$ of $X$ satisfying (a) such that $F$ and $F_T$ are related by local stabilizations and destabilizations.
    \item Given two triangulations $T$ and $T'$ satisfying (a) for the same Seifert surface $F$, the resulting surfaces $F_T$ and $F_{T'}$ are related by local stabilizations and destabilizations.
    \item Let $F$ and $F'$ be two Seifert surfaces related by the addition of a small tube $H$ as in Figure \ref{fig:add-tube}. Let $T$ be a triangulation satisfying both (a) and (b) for the surface $F'$ and tube $H$. Then $F_T$ and $F'_T$ are related by local stabilizations and destabilizations. 
\end{enumerate}

The proofs of (1) and (2) closely resemble the proof of the Reidemeister-Singer theorem in \cite[Theorem~1.2]{S1}.
Let $F$ be a free Seifert surface, so that $F$ determines a Heegaard splitting $X = \n(F) \cup (X\setminus \nu(F))$. Let $\G$ be the core graph of $F$ (and also of $\n(F)$), and let $T$ be a triangulation of $X$ such that $\G$ is a subcomplex.
In \cite{S1}, it is shown that by subdividing $T$ if needed, this Heegaard splitting is stably equivalent to the Heegaard splitting $X = \n(T^{(1)}) \cup (X \setminus\n(T^{(1)}))$. 
Therefore, there exists a sequence of unknotted handles $h_1, \ldots, h_k$ such that $\n(T^{(1)})$ is obtained from $\n(F)$ by attaching these handles in order. 
Note that $\n(T^{(1)})$ is equal to $\n(F)$ with solid handles attached along the edges in $T^{(1)}\setminus \G$, and attaching these solid handles amounts to a sequence of stabilizations of the Heegaard splitting. 
Following the same order, use the core arcs of each $h_i$ to attach a sequence of unknotted tubes to $F$. Then this is a sequence of local stabilizations, and the result is precisely $F_T$. 

For (2), let $F$ be a Seifert surface, let $\G$ be the core graph of $F$, and let $T$ and $T'$ be two triangulations of $X$ both satisfying (a). 
Let $T''$ be a common subdivision of $T$ and $T'$. 
Then it suffices to check that $F_T$ and $F_{T''}$ are related by local stabilizations and destabilizations. 
We check this simplex by simplex. Regardless of which edges of a given simplex are already in $\G$, subdividing this simplex introduces additional edges along which tubes must be attached to obtain $F_{T''}$. 
Attaching these tubes consists of a number of local stabilizations; see Figure \ref{fig:subdiv} below for an example.

\begin{figure}[htbp]
    \centering
    \def\svgwidth{.7\textwidth}
    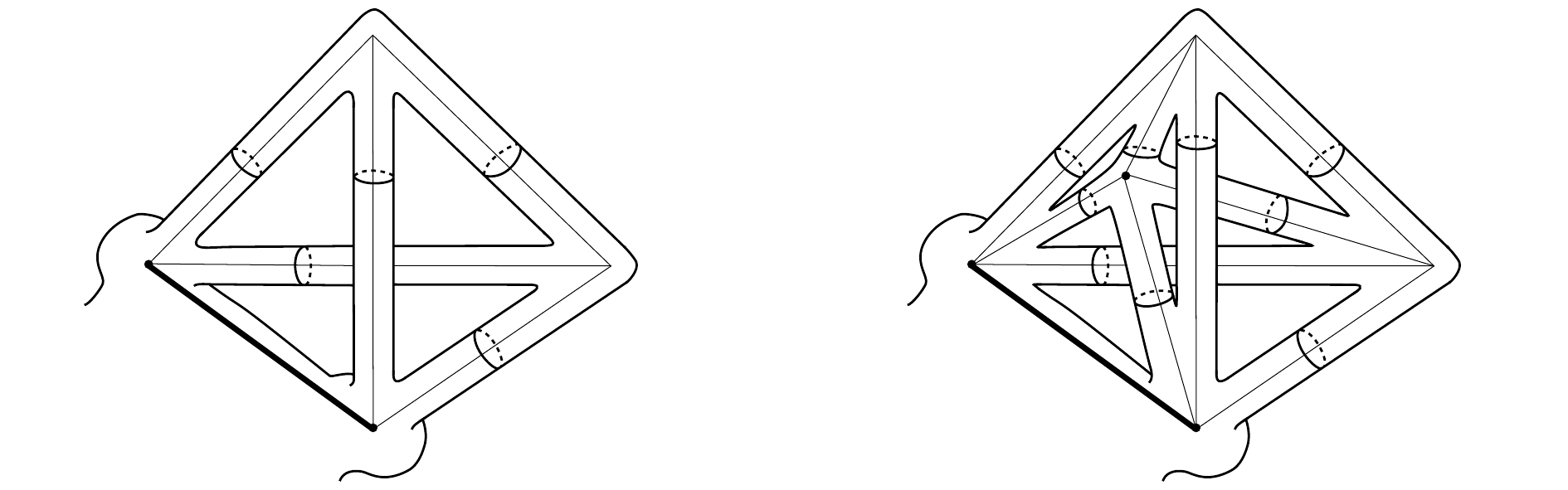
    \caption{Additional tubes coming from a subdivision of the triangulation. In this example, the bold edge in the bottom left is already contained in $\G$. Adding the new tubes can be viewed as a sequence of three local stabilizations.}
    \label{fig:subdiv}
\end{figure}

Finally, suppose that $F$ and $F'$ are Seifert surfaces related by adding a small tube $H$. Let $\G$ and $\G'$ be the core graphs of $F$ and $F'$.
Let $T$ be a triangulation satisfying (a) for $F'$ and its core graph $\G'$. 
Subdividing if needed, we may assume that $\G \subset \G' \subset T$. Viewing $H$ as a band and a dual band, the additional edges in $\G'$ consist of the cores of the these two bands. 
Now, consider the triangulation of $H$ pictured in Figure \ref{fig:H-triang}.
The cores of the bands forming $H$ are the union of the edges highlighted in purple. 
Further subdividing $T$ if needed, assume that this triangulation is a subtriangulation of $T$ restricted to $H$.
Assume for now that $T$ restricted to $H$ is precisely this triangulation. 

\begin{figure}[htbp]
    \centering
    \def\svgwidth{.5\textwidth}
    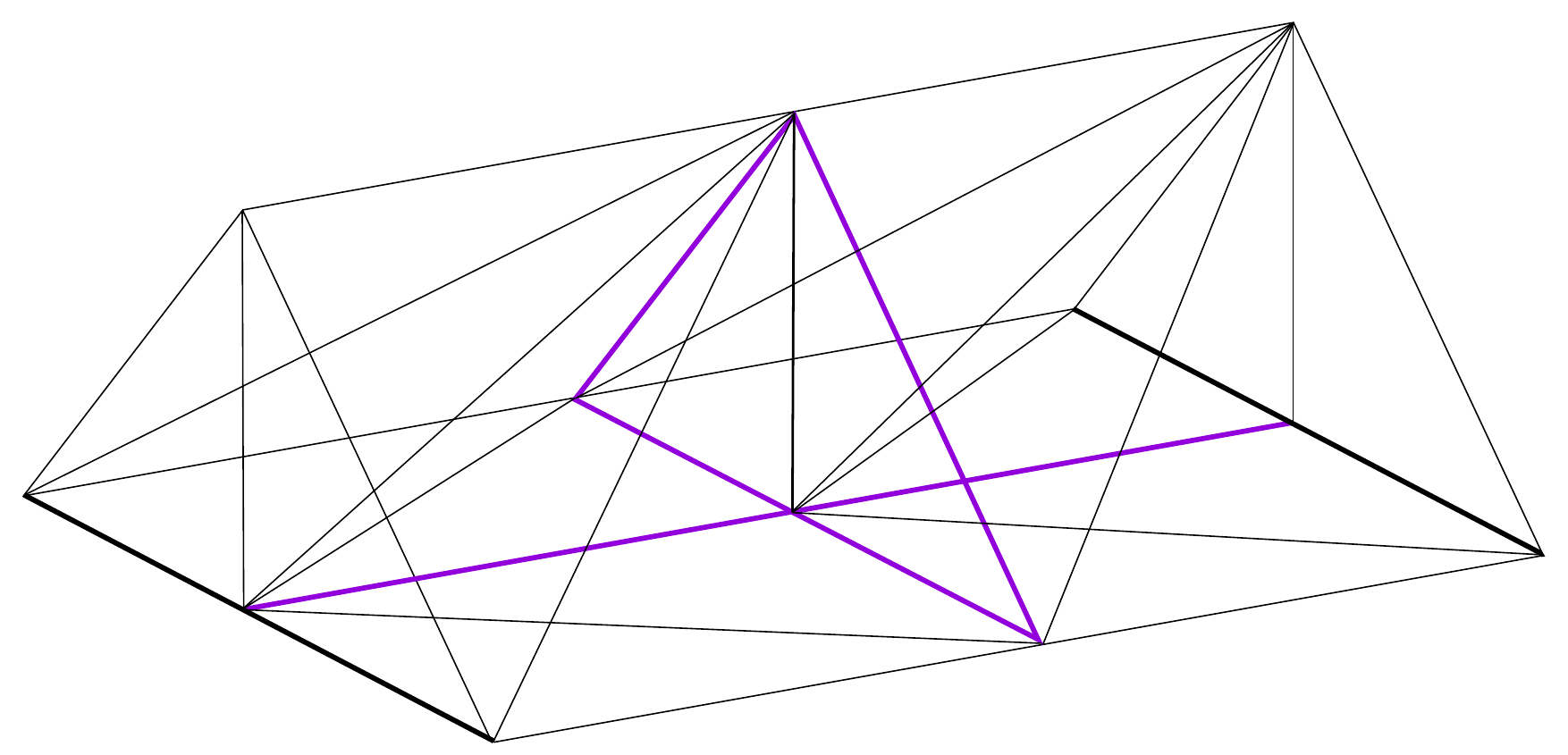
    \caption{A triangulation of $H$. The black bold edges are part of $\Gamma$ and the purple edges are $\G' \setminus \G$, namely the cores of the two bands forming $H$.}
    \label{fig:H-triang}
\end{figure}

The surfaces $F_T$ and $F'_T$ are formed by attaching tubes along the edges in $T^{(1)}\setminus \G$ and $T^{(1)}\setminus \G'$, respectively. 
Since $\G$ and $\G'$ only differ by the union of the edges forming the cores of the bands of $H$, the resulting surfaces $F_T$ and $F'_T$ only differ around where $H$ is attached. 
Using the triangulation of $H$ from Figure \ref{fig:H-triang}, we compare $F_T$ and $F_T'$ near $H$:

\begin{figure}[htbp]
    \centering
    \def\svgwidth{1.05\textwidth}
\begingroup%
  \makeatletter%
  \providecommand\color[2][]{%
    \errmessage{(Inkscape) Color is used for the text in Inkscape, but the package 'color.sty' is not loaded}%
    \renewcommand\color[2][]{}%
  }%
  \providecommand\transparent[1]{%
    \errmessage{(Inkscape) Transparency is used (non-zero) for the text in Inkscape, but the package 'transparent.sty' is not loaded}%
    \renewcommand\transparent[1]{}%
  }%
  \providecommand\rotatebox[2]{#2}%
  \newcommand*\fsize{\dimexpr\f@size pt\relax}%
  \newcommand*\lineheight[1]{\fontsize{\fsize}{#1\fsize}\selectfont}%
  \ifx\svgwidth\undefined%
    \setlength{\unitlength}{1951.3748193bp}%
    \ifx\svgscale\undefined%
      \relax%
    \else%
      \setlength{\unitlength}{\unitlength * \real{\svgscale}}%
    \fi%
  \else%
    \setlength{\unitlength}{\svgwidth}%
  \fi%
  \global\let\svgwidth\undefined%
  \global\let\svgscale\undefined%
  \makeatother%
  \begin{picture}(1,0.22695248)%
    \lineheight{1}%
    \setlength\tabcolsep{0pt}%
    \put(0,0){\includegraphics[width=\unitlength,page=1]{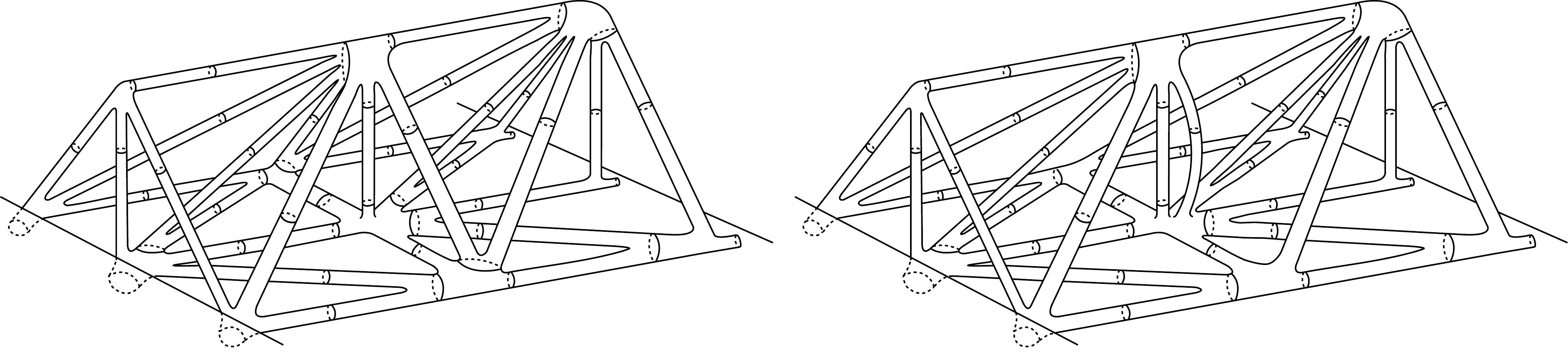}}%
    \put(0.21787351,0.00230238){\color[rgb]{0,0,0}\makebox(0,0)[lt]{\smash{\begin{tabular}[t]{l}$F_T$\end{tabular}}}}%
    \put(0.72434389,0.00230238){\color[rgb]{0,0,0}\makebox(0,0)[lt]{\smash{\begin{tabular}[t]{l}$F_T'$\end{tabular}}}}%
  \end{picture}%
\endgroup%

    \caption{The surfaces $F_T$ and $F_T'$ around $H$.}
    \label{fig:FTs}
\end{figure}

We see that $F_T$ and $F_T'$ differ by a single local stabilization. 
Earlier, we ignored additional edges of the triangulation $T$ restricted to $H$ coming from potential subdivisions. 
However, attaching tubes to $F$ and $F'$ for these additional edges amounts to doing some number of local stabilizations. Hence $F_T$ and $F_T'$ are related by local stabilizations.
\end{proof}

\begin{proposition}
    Let $(Y, \t, \bfw)$ be a double branched cover of $X$ along $L$. Let $H_1 = (\S_1, \bsa, \bsb, \bfw, \t)$ and $H_2 = (\S_2, \bsg, \bsd, \bfw, \t)$ be two multi-pointed real Heegaard diagrams for $(Y, \t, \bfw)$ such that $\S_1/\t$ and $\S_2/\t$ are both connected and orientable, and $[\S_1/\t] = [\S_2/\t]$ in $H_2(X, L ; \bbZ)$. 
    Then they are related by a series of Heegaard diagrams with orientable quotients all in that same relative homology class. 
\end{proposition}

\begin{proof}
    By Theorem 1 of \cite{N1}, since both $H_1$ and $H_2$ determine the zero framing on the fixed set $C \subset Y$, the two Heegaard splittings can be connected by a series of free stabilizations and diffeomorphisms. 
    There are two types of free stabilizations; we refer to these as orientable and nonorientable free stabilizations, depending on the orientability of the quotient of the stabilized surface. 
    An orientable free stabilization has the effect of performing a local stabilization on the quotient. 
    Since $[\S_1 / \t] = [\S_2 / \t]$, \Cref{seifert-surfaces} says they are related by a sequence of local stabilizations and destabilizations. 
    Upstairs, this means that $\S_1$ and $\S_2$ are related by orientable free stabilizations. 
    Once the Heegaard splitting is fixed, the sets of attaching curves can be transformed into each other by real isotopies and real handleslides as usual. 
\end{proof}

To show that this grading is well-defined, it therefore suffices to show invariance under real handleslides, real isotopies, and free orientable stabilizations of the real Heegaard diagram. 

\begin{proposition}
    The grading $\grh$ on $\widehat{\HFR}(Y, \t, \bfw)$ depends only on the orientation and ordering of the link components and the primitive homology class $S \in H_2(X, L ; \bbZ)$. 
\end{proposition}

\begin{proof}
    Let $(\S, \bsa, \bsb, \bfw, \t)$ be a real Heegaard diagram for $(Y, \t, \bfw)$ with $[\S/\t] = S$. Let $g$ denote the genus of $\S$ and let $l$ be the number of link components, so that there are $m = g + l-1$ many alpha curves. 
    Order and orient the alpha curves so that they induce the orientation of $\bbT_\a$ specified in the definition of $\grh$. 
    In Section 5 of \cite{GM1}, it is shown that real Heegaard moves give rise to maps inducing chain homotopy equivalences between the chain complexes associated to each real Heegaard diagram. 
    We show that the absolute grading on $\widehat{\CFR}$ is preserved under these maps. 
    
    A real isotopy does not change the homology classes of the alpha curves, so the canonical orientation of the isotoped alpha curves agrees with the orientation induced by the isotopy. 
    Invariance of $\grh$ then follows from \Cref{maslov} and the fact that the map showing invariance of real Heegaard Floer homology under real isotopies counts disks with Maslov index zero. 

    Let $(\S, \bsa', \bsb', \bfw, \t)$ be the diagram obtained by real handlesliding $\a_1$ over $\a_2$. 
    Following the proof of Theorem 1.2 in \cite{Per1}, the symplectic form $\omega$ can be chosen so that $\bbT_\a$ and $\bbT_\a'$ are Hamiltonian isotopic via $\Psi_t$. 
    By the construction of this isotopy, $\a_1' = \Psi_1(\a_1)$ is isotopic to $\a_1 \pm \a_2$ (where the sign depends on the orientation of $\a_2$). 
    It is straightforward to verify that $\{\a_1 \pm \a_2, \a_2, \ldots, \a_{m}\}$ induces the canonical orientation on $\bbT_\a'$. 
    The homotopy equivalence $\G_{\Psi_t}:\widehat{\CFR}(\S, \bsa, \bsb,\bfw, \t) \to \widehat{\CFR}(\S, \bsa', \bsb', \bfw, \t)$ is defined by counting disks with Maslov index zero, so it preserves $\grh$. 

    Finally, let $(\S_+, \bsa^+, \bsb^+, \bfw, \t)$ be the diagram obtained by a free orientable stabilization.
    This diagram is obtained by taking the connect sum with two standard genus 1 Heegaard diagrams for $S^3$. 
    The involution $\t$ is extended over the new handles by sending the alpha curve on one of the diagrams to the beta curve on the other. 
    Label the two new alpha curves $\a_{m+1}$ and $\a_{m+2}$, and label the new beta curves by $\b_{m+1} = \t(\a_{m+1})$ and $\b_{m+2} = \t(\a_{m+2})$. Let $p$ be the unique intersection point between $\a_{m+1}$ and $\b_{m+2}$, so that $\t(p)$ is the intersection point between $\a_{m+2}$ and $\b_{m+1}$. 
    In the quotient $\S/\t$, the images of these curves intersect exactly once, and are disjoint from the images of all the other alpha curves. Orient them so that the images in the quotient $\S/\t$ satisfy $\a_{m+1} \cdot \a_{m+2} = 1$. 
    Then the resulting set of ordered and oriented alpha curves determines the canonical orientation of $\bbT_\a^+$. 
    The orientation of $\S_+/\t$ is still determined by the orientation and ordering of $L$. Let $M^+ = \text{Sym}^{m+2}(\S_+)$, and let $\widehat{(M^+)^R}$ be the fixed point set minus the divisor $\bfw \times \text{Sym}^{m+1}(\S_+)$. 
    There is a 1-1 correspondence between points in $\bbT_\a \cap \widehat{M^R}$ and $\bbT_\a^+ \cap \widehat{(M^+)^R} $: a point $\bfx \in \bbT_\a \cap \widehat{M^R}$ corresponds to the point $\bfx \times \{p, \t(p)\} \in \bbT_\a^+ \cap \widehat{(M^+)^R}$. 
    Using (\ref{int-sign}) to compute the signs of these intersection points, we have that $\sign(\bfx) = \sign(\bfx \times \{p ,\t(p)\})$.  
\end{proof}

Although this grading may depend on the ordering and orientation of the link components, it is straightforward to show that it is unchanged if we replace $L$ with its reverse $rL$ and reverse the ordering. 

\begin{proposition}
    Let $L$ be an oriented link and fix a primitive class $S \in H_2(X, L ; \bbZ)$ satisfying $\partial S = [L]$. 
    Let $(Y, \t)$ be the double branched cover of $X$ along $L$ corresponding to the image of $S$ in $H_2(X, L;\bbZ/2)$. 
    Fix an ordering of the link components and a set of basepoints $\bfw \subset Y$.
    Then the absolute $\bbZ/2$ grading defined on $\widehat{\HFR}(Y, \t, \bfw)$ using the ordered and oriented link $L$ and class $S$ is the same as the one using the link $rL$ with the ordering reversed and the class $-S$.  
    \label[proposition]{reverse-it}
\end{proposition}

\begin{proof}
    Take a Seifert surface $\S'$ representing $S$ and construct a real Heegaard splitting of $(Y, \t, \bfw)$ following \cite{GM1}. 
    Choose alpha curves and let $\t$ determine the beta curves. 
    Orient $\S'$ so that the orientation induced on its boundary agrees with the one on $L$. 
    Along with the ordering of the link components, this determines an orientation of $\widehat{M^R}$. 
    Order and orient the alpha curves so that they induce the canonical orientation of $\bbT_\a$. 
    Using the involution $\t$, we get an ordering and orientation of the beta curves as well. Let $\mathbf{x} \in \bbT_\a \cap \widehat{M^R}$, so that $\mathbf{x}$ is of the form 
    \[\bfx = \{z_1, \t(z_1), \ldots, z_k, \t(z_k), c_1, \ldots, c_{g - 2k}\} \in \text{Sym}^{k}(\S') \times \text{Sym}^{m - 2k}(\widehat{C}).\]
    If we reverse the orientation of every component of $L$, observe that $\overline{\S'}$ is a Seifert surface for $rL$, where $\overline{\S'}$ denotes $\S'$ with the opposite orientation. 
    Thus, using $rL$ (with the order reversed) to orient the Lagrangians will affect the orientation on $\widehat{M^R}$; denote the result $\widehat{\overline{M^R}}$. 
    Following the definition of $\grh$, the correct ordering and orientation of the alpha curves using the new ordering and orientation of $L$ will remain the same if and only if $(-1)^{h+l-1} = 1$, where $l$ is the number of link components and $h$ is the genus of $\S'$. 
    
    Let $\overline{\bfx}$ denote the corresponding generator in $\bbT_\a \cap \widehat{\overline{M^R}}$. 
    To distinguish the orientations we are using to talk about the sign of an intersection point, a point with an overline always refers to that point with the orientations induced by $rL$ and the reverse ordering. 
    Then to compute the sign of $\overline{\bfx}$ as an intersection point of $\bbT_\a \cap \widehat{\overline{M^R}}$ using (\ref{int-sign}), we write it as follows, following the conventions in \cite{GM1}:
    \[\overline{\bfx} = \{\overline{\t(z_1)}, \overline{z_1}, \ldots, \overline{\t(z_k)}, \overline{z_k}, \overline{c_{m-2k}}, \ldots, \overline{c_{1}}\} \in \text{Sym}^{k}(\overline{\S'}) \times \text{Sym}^{m - 2k}(\widehat{rC}).\]
    This is because if we pull back the orientation on $\overline{\S'}$ to $\S$, the roles of $\S^+$ and $\S^-$ swap. 
    Additionally, reversing the ordering and orientation of each link component means we encounter the $c_i$ in the reverse order.
    We have that $\e(\overline{\t(z_i)}) = \e(z_i)$ and $\e(c_i) = -\e(\overline{c_i})$. Let $\e \in \{\pm 1\}$ denote the sign of the permutation $\e(r, \s, s)$ for $\bfx$, so that the sign of the corresponding permutation for $\overline{\bfx}$ is $ (-1)^{h+l-1}\e$. 
    It follows that $\text{sgn}(\bfx) = \text{sgn}(\overline{\bfx})$.  
\end{proof}

\mycomment{
\begin{remark}
    In general, this grading does indeed depend on the ordering and orientation of the link components. Varying these choices for the Hopf link in $S^3$ gives rise to different absolute $\bbZ/2$ gradings on $\widehat{\CFR}(Y, \t)$. It is currently unknown whether or not the grading actually depends on the choice of homology class $S \in H_2(X, L ; \bbZ)$.
\end{remark}}

We conclude this section by computing $\grh$ for the double branched cover of the Hopf link in $S^3$ with its two possible choices of orientations. 

\begin{example} Let $H^\pm$ denote the Hopf link with linking number $\pm 1$, and consider the diagrams for $H^\pm$ in Figure \ref{fig:hopf-links}. 

\begin{figure}[htbp]
    \centering
    \def\svgwidth{.6\textwidth}
    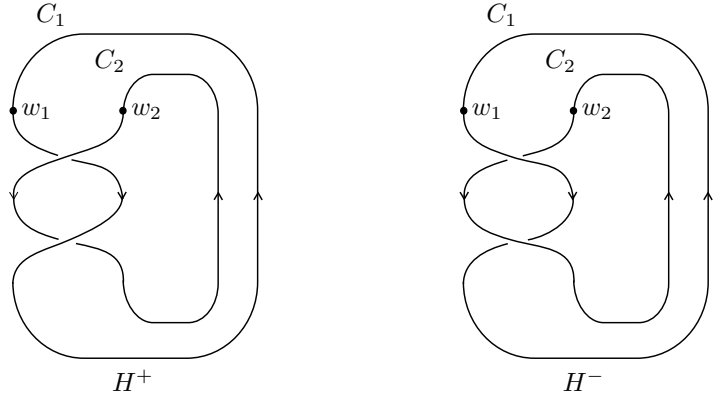
    \caption{The Hopf link with its two choices of orientations.}
    \label{fig:hopf-links}
\end{figure}

We have additionally fixed an ordering of the components and a basepoint on each component. 
Applying Seifert's algorithm, we obtain Seifert surfaces $(\S')_\pm$ for $H^\pm$. The curves $\a_1^\pm$ and $\g^\pm$ determine a basis for $H_1(\widehat{(\S')_\pm})$, where $\widehat{(\S')_\pm}$ is the surface obtained by capping $(\S')_\pm$ off with a connected planar domain. 

\begin{figure}[htbp]
    \centering
    \def\svgwidth{.62\textwidth}
    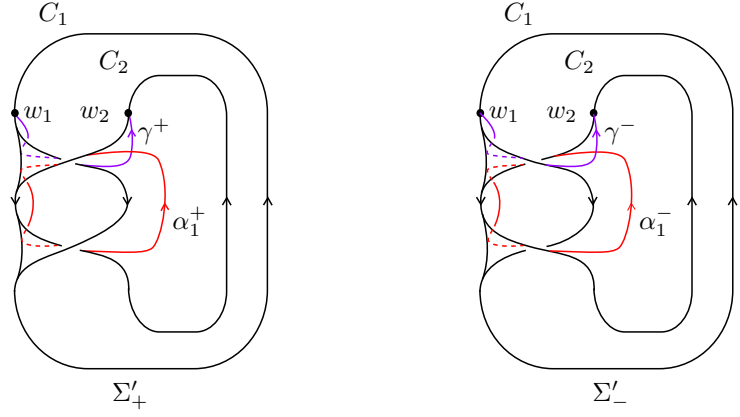
    \caption{Seifert surfaces $\S_\pm'$ for the two Hopf links.}
    \label{fig:hopf-link-seifert-surface}
\end{figure}

Tracing through the isomorphisms
\begin{align*}
    H_1(U_\a^\pm, \bfw^\pm) &\cong H_1((\S')_\pm)\oplus \ker(H_0(\bfw^\pm) \to H_0(U_\a^\pm)) \\
    &\cong H_1((\S')_\pm)\oplus H_1(\widehat{(\S')_\pm}, (\S')_\pm)\\
    &\cong H_1(\widehat{(\S')_\pm}), 
\end{align*}
the curve $\g^\pm$ first gets mapped to a curve $\z^\pm$ which connects $w_1$ to $w_2$ and is contained entirely on the connected planar domain. 
This then gets mapped to the curve $\z^\pm \cup -\g^\pm$ on $\widehat{(\S')_\pm}$. 
With the orientations above, $\a_1^\pm \cdot \g^\pm = -1$, so that $\langle \a_1^\pm, \z^\pm \cup -\g^\pm\rangle $ is a symplectic basis for $H_1(\widehat{(\S')_\pm})$. 
Letting $\a_2^\pm$ be a simple closed curve encircling $w_2$ and satisfying $\a_2^\pm \cdot \g^\pm = 1$, we obtain the two Heegaard diagrams in \Cref{fig:hopf-link-heeg-diags}, which represent $(\S_\pm, \bsa^\pm, \bsb^\pm, \bfw^\pm, \t^\pm)$. 
Furthermore, the pictured orientations in \Cref{fig:hopf-link-heeg-diags} determine the orientation of $\bbT_\a^\pm$ used to compute $\grh$. 

\begin{figure}[htbp]
    \centering
    \def\svgwidth{.9\textwidth}
    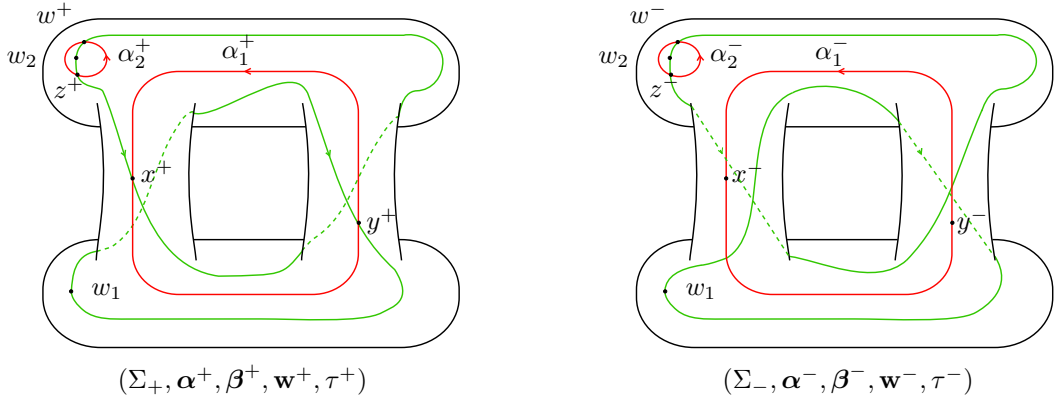
    \caption{Heegaard diagrams for $(\S_2(H^\pm), \t^\pm)$ built from the Seifert surfaces $\S'_\pm$. The fixed point sets are drawn in green. The beta curves are not shown, but can be obtained by applying $\tau^\pm$ to the alpha curves.}
    \label{fig:hopf-link-heeg-diags}
\end{figure}

We first consider the grading on $\widehat{\HFR}(\S_2(H^+), \t^+, \bfw^+)$. 
Looking at the left image in \Cref{fig:hopf-link-heeg-diags}, there are four intersection points in $\bbT_\a^+ \cap \widehat{(M^+)^R}$:
\[x^+z^+, x^+w^+, y^+z^+, y^+w^+\]
The points $x^+z^+$ and $ x^+w^+$ both lie in one of the two real Spin$^c$-structures, while $y^+z^+$ and $y^+w^+$ both lie in the other; call these real Spin$^c$-structures $\mathfrak{s}_1^+$ and $\mathfrak{s}_2^+$ respectively. 
Using (\ref{int-sign}) to compute the grading, we have 
\[\grh(x^+z^+) = 0, \hspace{.2cm} \grh(x^+w^+) = 0, \hspace{.2cm} \grh(y^+z^+) = 0, \hspace{.2cm} \grh(y^+w^+) = 1.\]
Therefore, 
\[\widehat{\HFR}(\S_2(H^+), \t^+, \mathfrak{s}_1^+, \bfw^+) \cong \bbF^2 = \bbF_{(0)} \oplus \bbF_{(0)}\]
and 
\[\widehat{\HFR}(\S_2(H^+), \t^+, \mathfrak{s}_2^+, \bfw^+) \cong \bbF^2 = \bbF_{(0)} \oplus \bbF_{(1)},\]
where the subscripts denote the $\bbZ/2$ grading of the generator of each summand. 

Similarly, we may compute $\grh$ for $\widehat{\HFR}(\S_2(H^-), \t^-, \bfw^-)$. This time, 
\[\widehat{\HFR}(\S_2(H^-), \t^-, \mathfrak{s}_1^-, \bfw^-) \cong \bbF^2 = \bbF_{(0)} \oplus \bbF_{(1)}\]
and 
\[\widehat{\HFR}(\S_2(H^-), \t^-, \mathfrak{s}_2^-, \bfw^-) \cong \bbF^2 = \bbF_{(1)} \oplus \bbF_{(1)}.\]

As this example illustrates, $\grh$ generally depends on the orientation of the link components. 
\label[example]{hopf-link-ex}
\end{example}

\section{The Euler Characteristic}
\subsection{The sign of the Euler characteristic}
Let $L$ be an oriented link and fix an ordering of the link components. 
Fix a primitive class $S \in H_2(X, L ; \bbZ)$ satisfying $\partial S = [L]$. 
Let $(Y, \t, \bfw)$ be the double branched cover of $X$ along $L$ corresponding to the image of $S$ in $H_2(X, L ; \bbZ/2)$, where $\t$ is the branching involution.  
Using the absolute $\bbZ/2$ grading $\grh$ defined in Section 3, we may fix the sign of the hat version of the Euler characteristic $\widehat{\c}(\widehat{\HFR}(Y, \t, \bfw))$. From \Cref{reverse-it}, the sign will be unchanged if we both replace $L$ with the reverse $rL$ and reverse the ordering of the components. 
In fact, the Euler characteristic in each real $\text{Spin}^c$-structure is actually independent of the ordering of the link components. 

\begin{proposition}
    For each real $\text{Spin}^c$-structure on $(Y, \t, \bfw)$, the sign of the hat version of the Euler characteristic $\widehat{\c}(\widehat{\HFR}(Y, \t, \mathfrak{s}, \bfw))$ is independent of the ordering of the link components.
    \label[proposition]{order-inv}
\end{proposition}

\begin{proof}
    It suffices to check that the sign is unchanged under exchanging the labels of two components. 
    Let $I$ and $J$ be two link components. 
    Fix some ordering of the components, and suppose without loss of generality that $I = I_i$ and $J = I_j$ for $i < j$ with respect to this ordering. 
    Take a representative $\S'$ of $S$ and construct a real Heegaard diagram for $(Y, \t, \bfw)$ as detailed in Section 3. 
    In particular, if $h$ is the genus of $\S'$ and $l$ is the number of link components, the genus of the surface in the real Heegaard splitting is $g = 2h + l-1$. Let $m = g + l-1$ be the number of alpha curves. 
    Then $l-1$ of these alpha curves, $\a_{g+1}, \ldots, \a_{g +l-1}$, are such that each $\a_{g+n}$ is a small simple closed curve encircling the basepoint on the $(n+1)$-th component of $L$, intersecting it in precisely two points with opposite sign. 
    The ordering and orientations of the components determine an orientation of $\widehat{M^R}$, and the orientations determine one of $\bbT_\a$. 
    If we replace this ordering with one that simply exchanges the positions of $I$ and $J$, then the orientation of $\bbT_\a$ doesn't change, but the orientation of components of $\widehat{M^R}$ will. 

    More concretly, fix $k \in \{0, \ldots, h + l-1\}$ and consider the stratum
    \[\text{Sym}^{2k}(\S') \times \text{Sym}^{m - 2k}(\widehat{C})\]
    of $\widehat{M^R}$. 
    This breaks up further into components of the form 
    \[\text{Sym}^{2k}(\S') \times (I_1^{n_1} \times \ldots \times I_{l}^{n_l})/\sim, \]
    where $n_1 + \ldots + n_l = m - 2k$ and we are quotienting by the symmetric group action on the second factor. 
    Exchanging the positions of $I$ and $J$ then changes the orientation on this component by a factor of
    \begin{equation}
    (-1)^{n_in_j + (n_i+n_j)\sum_{i < k < j}n_k}.
    \label{swap-i-j}
    \end{equation}
    Suppose that a component of this form changes orientation when we orient $\widehat{M^R}$ with respect to the new ordering. 
    Then the sign in (\ref{swap-i-j}) must be equal to $-1$, and we see that this requires at least one of $n_i$ and $n_j$ to be odd. Let 
    \[\bfx = \{z_1, \t(z_1), \ldots, z_k, \t(z_k), c_1^1, \ldots, c_{n_1}^1, \ldots, c^l_{1}, \ldots c^l_{n_l}\} \in \text{Sym}^{2k}(\S') \times (I_1^{n_1} \times \ldots \times I_{l}^{n_l})/\sim\]
    be a generator. 
    As usual, we require these points to be in order following the orientation and initial ordering of $L$. 
    Here, $c^p_n$ denotes the $n$-th point encountered on the $p$-th link component. 

    Suppose first that $n_j$ is odd. 
    Since $j > 1$, there is an alpha curve $\a_{g + (j-1)}$ which encircles the basepoint on $L_j$ and intersects this component in exactly two points $\{p^\pm\}$ with opposite sign. 
    By the construction of this real Heegaard diagram, we must have that either $c^j_1$ or $c^j_{n_j}$ lies on $\a_{g + (j-1)}$. 
    Without loss of generality, suppose $c^j_1 = p^+$. Then the point 
    \[\bfy = \{z_1, \t(z_1), \ldots, z_k, \t(z_k), c_1^1, \ldots, c_{n_1}^1, \ldots, c_2^j, \ldots, p^-, \ldots, c^l_{1}, \ldots c^l_{n_l}\}\]
    is another generator, and $\text{sgn}(\bfx) = -\text{sgn}(\bfy)$. 
    Moreover, these generators belong to the same real Spin$^c$-structure. 

    Otherwise, if $n_j$ is even, we must have that $n_i$ is odd. If $i > 1$, simply apply the argument above to $\a_{g + (i -1)}$. 
    So, suppose $i = 1$. Looking at (\ref{swap-i-j}), we necessarily have that $j > 2$ and there is some $1 < k < j$ with $n_k$ odd. 
    If this were not the case, then the orientation on $\text{Sym}^{2k}(\S') \times (I_1^{n_1} \times \ldots \times I_{l}^{n_l})/\sim$ would remain the same when we exchange the labels of $I$ and $J$. 
    Then apply the same argument to this $\a_{g + (k-1)}$. 

    In any case, we see that signs of intersection points which are contained in components of $\widehat{M^R}$ that change sign upon exchanging the labels of $I$ and $J$ will always cancel in pairs. 
    In addition, the points in each cancelling pair lie in the same real Spin$^c$-structure.
    Consequently, the Euler characteristic of each $\widehat{\HFR}(Y, \t, \mathfrak{s}, \bfw)$ is independent of the ordering of the link components.   
\end{proof}

Note that this allows us to upgrade \Cref{reverse-it}: the isomorphism class of $\widehat{\HFR}(Y, \t, \mathfrak{s}, \bfw)$ as a $\bbZ/2$ graded $\bbF$-module is independent of the ordering of the link components. 
Consequently, the grading is unchanged if we replace $L$ with its reverse. 

\begin{remark}
    The sign of the Euler characteristic does depend on the orientation of the link components, in the case that $L$ has more than one component. 
    The double branched cover of the Hopf link in $S^3$ provides an example where different choices of orientations give rise to different values of $\widehat{\c}(\widehat{\HFR}(Y, \t, \bfw))$. 
    Let $H^\pm$ the Hopf link with linking number $\pm 1$. 
    In \Cref{hopf-link-ex}, we computed $\grh$ for $\widehat{\HFR}(\S_2(H^\pm), \t^\pm, \bfw^\pm)$. 
    Summing the signs of intersection points, we can compute the Euler characteristics $\widehat{\c}$ for each of these: 
    \[\widehat{\c}(\widehat{\HFR}(\S_2(H^+), \t^+, \bfw^+)) = 2\]
    \[\widehat{\c}(\widehat{\HFR}(\S_2(H^-), \t^-, \bfw^-)) = -2.\]
    It is currently unknown whether or not the sign of the Euler characteristic actually depends on the choice of primitive homology class $S \in H_2(X, L ;\bbZ)$.
\end{remark}

\subsection{Double branched covers of links}
We now specialize to the case of links $L$ in $S^3$. 
As outlined in the introduction, real Heegaard Floer homology admits a decomposition by real Spin$^c$-structures. 
For double branched covers $\S_2(L)$ over links in $S^3$, equipped with their branching involution $\t_L$, each Spin$^c$ structure has a unique real Spin$^c$-structure, and so we can consider the Euler characteristic in each Spin$^c$-structure $\mathfrak{s}$:
\[\cs(L) = \widehat{\c}(\widehat{\HFR}(\S_2(L), \t_L, \mathfrak{s}, \bfw)).\]
Using $\grh$, the signs of each of these are fixed, along with the sign of the total Euler characteristic. 
By \Cref{order-inv}, these are invariants of the oriented link $L$. 

\begin{definition}
    Let $L \subset S^3$ be a link, and let $(\S_2(L), \t_L)$ be the double branched cover of $S^3$ along $L$. Fix a basepoint on each component. Define 
    \[\ctot(L) = \widehat{\c}(\widehat{\HFR}(\S_2(L), \t_L, \bfw )),\]
    and for each $\mathfrak{s \in }\text{Spin}^c(\S_2(L))$, define  
    \[\cs(L) = \widehat{\c}(\widehat{\HFR}(\S_2(L), \t_L, \mathfrak{s}, \bfw)),\]
    where $\widehat{\HFR}(\S_2(L), \t_L, \bfw)$ is given the grading $\grh$ from \Cref{main-theorem2}.
\end{definition}

Finally, we show that $\ctot$ can be expressed in terms of the Alexander polynomial. 

\begin{proof}[Proof of \Cref{alex-poly}]
    We show that $\ctot(L)$ and $2^{l-1}\D_L(i, \ldots, i)$ both take the same value on the unknot and satisfy the same oriented skein relation. 

    The double branched cover of the unknot $U \subset S^3$ admits a genus zero real Heegaard diagram with no alpha curves.
    Then $\widehat{\CFR}(\S_2(U), \t_U)$ is generated by a single point, which trivially has sign $1$, so that $\ctot(U) = 1$. 
    
    Next, let $L_0, L_+, L_-$ be three oriented links in $S^3$ admitting diagrams obtained from braid closures that differ only at one crossing, where they look as in \Cref{fig:skein-relns}. 

    \begin{figure}[htbp]
    \centering
    \def\svgwidth{.8\textwidth}
\begingroup%
  \makeatletter%
  \providecommand\color[2][]{%
    \errmessage{(Inkscape) Color is used for the text in Inkscape, but the package 'color.sty' is not loaded}%
    \renewcommand\color[2][]{}%
  }%
  \providecommand\transparent[1]{%
    \errmessage{(Inkscape) Transparency is used (non-zero) for the text in Inkscape, but the package 'transparent.sty' is not loaded}%
    \renewcommand\transparent[1]{}%
  }%
  \providecommand\rotatebox[2]{#2}%
  \newcommand*\fsize{\dimexpr\f@size pt\relax}%
  \newcommand*\lineheight[1]{\fontsize{\fsize}{#1\fsize}\selectfont}%
  \ifx\svgwidth\undefined%
    \setlength{\unitlength}{1252.40226746bp}%
    \ifx\svgscale\undefined%
      \relax%
    \else%
      \setlength{\unitlength}{\unitlength * \real{\svgscale}}%
    \fi%
  \else%
    \setlength{\unitlength}{\svgwidth}%
  \fi%
  \global\let\svgwidth\undefined%
  \global\let\svgscale\undefined%
  \makeatother%
  \begin{picture}(1,0.25928316)%
    \lineheight{1}%
    \setlength\tabcolsep{0pt}%
    \put(0,0){\includegraphics[width=\unitlength,page=1]{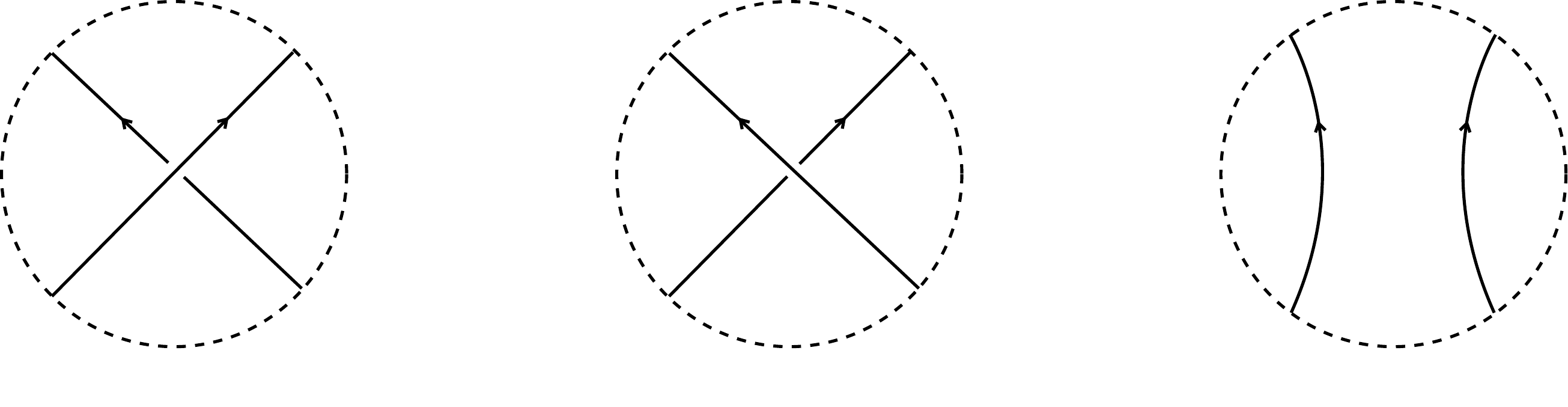}}%
    \put(0.09856877,0.00299212){\color[rgb]{0,0,0}\makebox(0,0)[lt]{\smash{\begin{tabular}[t]{l}$L_+$\end{tabular}}}}%
    \put(0.49591411,0.00299212){\color[rgb]{0,0,0}\makebox(0,0)[lt]{\smash{\begin{tabular}[t]{l}$L_-$\end{tabular}}}}%
    \put(0.8776826,0.00299212){\color[rgb]{0,0,0}\makebox(0,0)[lt]{\smash{\begin{tabular}[t]{l}$L_0$\end{tabular}}}}%
  \end{picture}%
\endgroup%

    \caption{Oriented skein relations.}
    \label{fig:skein-relns}
    \end{figure}
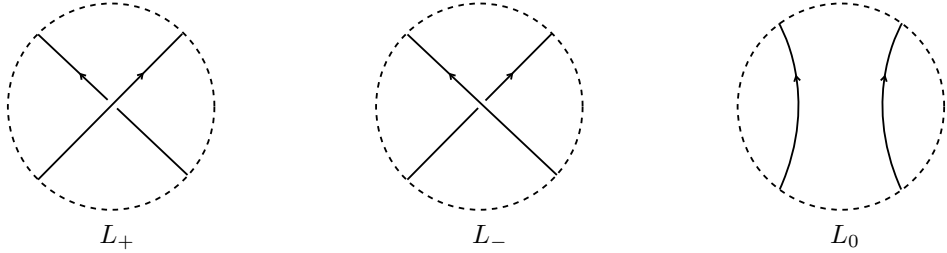
    
    The multivariate polynomial $\D_L(t_1, \ldots, t_l)$ and Conway-normalized Alexander polynomial $\D_L(t)$ of an $l$-component oriented link are related by 
    \[(t^{-1/2} - t^{1/2})^{l-1}\D_L(t, \ldots, t) = \D_L(t).\]
    Using this and the standard skein relation for $\D_L(t)$, it is straightforward to verify that $\D_L(i, \ldots,i)$ satisfies the following two skein relations: 
    \[\D_{L_+}(i, \ldots, i) - \D_{L_-}(i, \ldots, i) = \D_{L_0}(i, \ldots, i) \hspace{.25cm} \text{ if }\lvert L_\pm\rvert = \lvert L_0\rvert + 1,\] 
    \[\D_{L_+}(i, \ldots, i)-\D_{L_-}(i, \ldots, i) = -2\D_{L_0}(i,\ldots,i)\hspace{.25cm} \text{ if } \lvert L_\pm\rvert = \lvert L_0 \rvert -1.\]
    Multiplying through by the appropriate powers of two, we need to show that $\ctot$ satisfies the following two skein relations:
    \begin{equation} 
    \ctot(L_+) - \ctot(L_-) = 2\ctot(L_0) \hspace{.25cm} \text{ if }\lvert L_\pm\rvert = \lvert L_0\rvert + 1,
    \label[equation]{skein-reln-1}
    \end{equation}
    \begin{equation}
    \ctot(L_+) - \ctot(L_-) = -\ctot(L_0) \hspace{.25cm}\text{ if } \lvert L_\pm\rvert = \lvert L_0 \rvert -1.
    \label[equation]{skein-reln-2}
    \end{equation}
    \mycomment{
    We may assume that the crossing we are changing is the last generator in the braid word for $L_+$. 
    By inserting a cancelling pair of generators in the braid word if needed, we may also assume that this generator has previously appeared in the braid word. 
    Use Seifert's algorithm to construct Seifert surfaces $\S_+', \S_-',$ and $\S_-'$ for each of the links $L_+, L_-,$ and $L_0$ respectively. 
    Thicken the Seifert surfaces and take the alpha curves pictured in [FIG]. 
    Using the involution to obtain the beta curves gives real Heegaard diagrams for the double branched covers of these links equipped with their branching involutions. 
    These real Heegaard diagrams are identical away from the crossing that gets changed.
    To show the first relation (\ref{skein-reln-1}), label the link components and alpha curves as pictured in the leftmost image in [FIG?]. 
    Order and orient the alpha curves so that $\langle \a_1, \ldots, \a_{2h+l-1}, \a_{2h+l}, \ldots, \a_{2h+2(l-1)}\rangle$ determines the canonical orientation on $\bbT_\a^+$. Then the analogous ordering and orientation of the alpha curves on $\S_-$ also determines the canonical orientation on $\bbT_\a^-$. Removing the pair $\a_{2h+l-1}$ and $ \a_{2h+l}$ and keeping the ordering and orientation of the other curves the same results in a canonically oriented set of alpha curves for $\S_0$. Figure [FIG?] shows the real Heegaard diagrams for $\S_2(L_+), \S_2(L_-),$ and $\S_2(L_0)$ around the crossing that gets changed.}
    
    Let $b^+ = x_1 \cdot \ldots \cdot x_n \in B$ be a braid word such that the closure yields a diagram for $L_+$. 
    By cyclically permuting the letters of $b^+$ if necessary, we may assume that the crossing being changed is the one given by the last letter $x_n$. 
    Also, by inserting a canceling pair of letters, we may assume that that if $x_n = \s_i \in B$, then $\lvert x_m\rvert = \lvert\s_i\rvert$ for some $m < n$. 
    In other words, $x_n$ is not the first instance of that generator (or its inverse) in the braid word. 
    Then diagrams for $L_-$ and $L_0$ are given by closures of the braids $b^- = x_1 \cdot \ldots \cdot(-x_n)$ and $b^0 = x_1 \cdot \ldots \cdot x_{n-1}$ respectively. 

    Apply Seifert's algorithm to these diagrams to obtain Seifert surfaces $\S_+', \S_-'$, and $\S_0'$ for $L_+, L_-$, and $L_0$ respectively. 
    As in Section 3 of \cite{GM1}, thicken these surfaces to obtain handlebodies and choose alpha curves on the boundaries $\S_+, \S_-$, and $\S_0$ as follows: for each generator $\s \in B$, suppose that $\lvert x_i\rvert$ is the first instance of that generator in the braid word. 
    Then for each subsequent $x_j$, $j > i$, such that $\lvert x_j \rvert = \s_i$, take an alpha curve which passes over the two handles of the surface coming from each of these generators. 

    We first verify the first relation. Suppose that strands $s$ and $s+1$ are involved in the crossing $x_n$. 
    Fix the ordering of the link components so that strand $s$ belongs to component $1$ and strand $s+1$ belongs to component 2. 
    Fix basepoints $w_1$ and $w_2$ at the top of strands $s$ and $s+1$. Order the rest of the components arbitrarily, and fix basepoints for each remaining component at the top of one of the strands belonging to the component. 
    Around each basepoint $w_j, j \geq 2$, take an alpha curve which encircles that basepoint in a small simple closed curve. 
    Along with the alpha curves described earlier, this information determines a pointed real Heegaard diagram for $(\S_2(L_+), \t^+, \bfw^+)$. 
    Let $\a_1^+$ be the alpha curve that runs over the handle coming from $x_n$, and let $\a_2^+$ be the alpha curve encircling $w_2$. 
    Orient these two curves as in the leftmost image in \Cref{fig:skein-reln-heeg1}. Orient and order the remaining alpha curves so that $\langle \a_1^+, \ldots, \a_{2h+2(l-1)}^+ \rangle$ determines the canonical orientation of $\bbT_\a^+$. 

    Label the components of $L_-$ and fix basepoints on $L_-$ analogously, and take the same alpha curves as well. 
    Then we also obtain a pointed real Heegaard diagram for $(\S_2(L_-), \t^-, \bfw^-)$. 
    For $L_0$, we can take analogues of the alpha curves $\a_j^+$ for $j \geq 3$. 
    Let component 1 be the component of $L_0$ which contains strands $s$ and $s+1$, and keep the basepoint $w_1$. 
    Label the remaining components of $L_0$ in the same order as $L_\pm$, and keep the remaining basepoints. 
    This determines a pointed real Heegaard diagram for $(\S_2(L_0), \t^0, \bfw^0)$. 
    Furthermore, the bases $\langle \a_1^-, \ldots, \a_{2h+2(l-1)}^- \rangle $ and $\langle \a_3^0, \ldots, \a_{2h+2(l-1)}^0\rangle$ determine the canonical orientations on $\bbT_\a^+$ and $\bbT_\a^0$. The resulting Heegaard diagrams then only differ around the handle coming from the generator $\pm x_n$, where they look as in \Cref{fig:skein-reln-heeg1}. 

    \begin{figure}[htbp]
    \centering
    \def\svgwidth{.9\textwidth}
    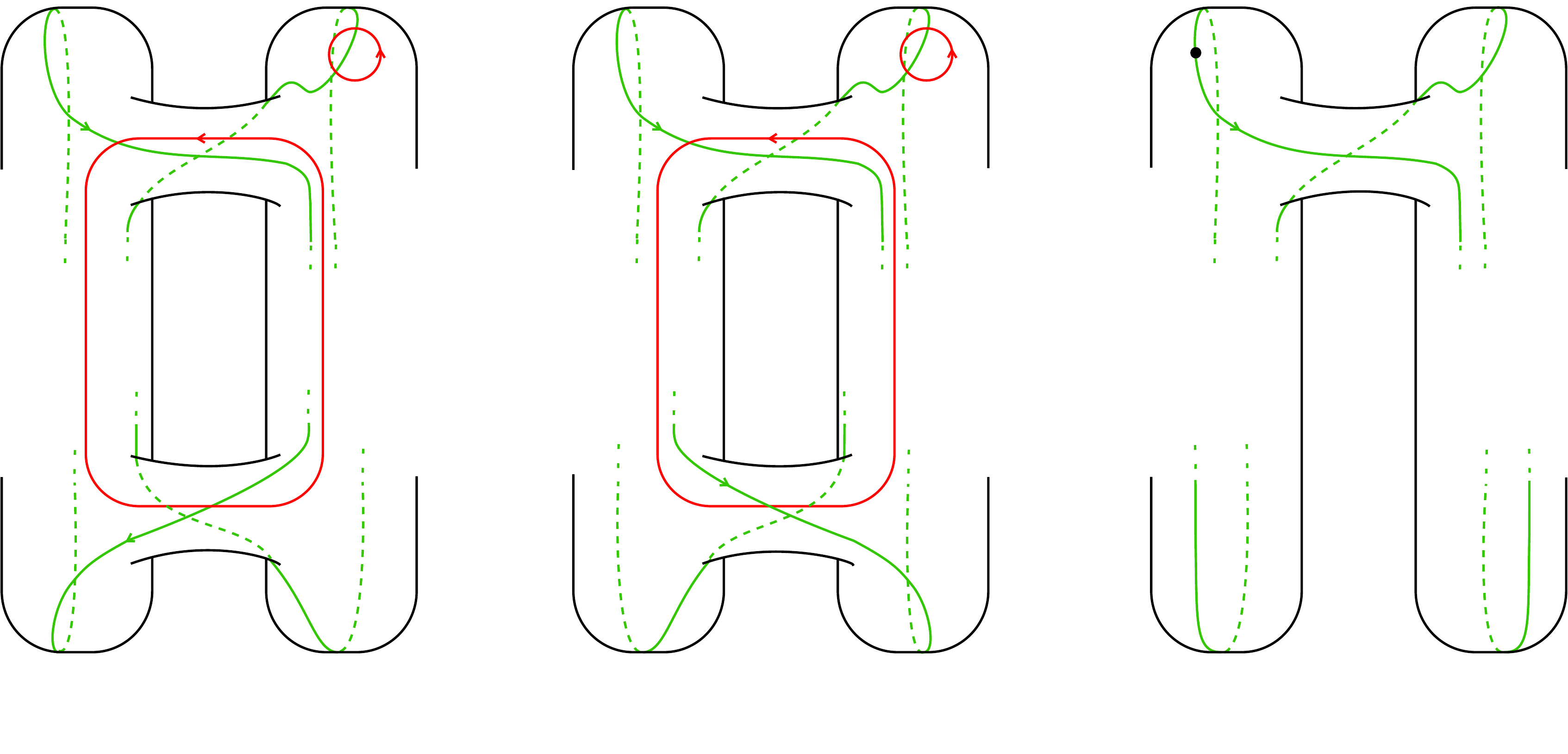
    \caption{Heegaard diagrams for $(\S_2(L_+), \t^+, \bfw^+), (\S_2(L_-), \t^-, \bfw^-)$, and $(\S_2(L_0), \t^0, \bfw^0)$ when $\lvert L_\pm\rvert=\lvert L_0\rvert + 1$.}
    \label{fig:skein-reln-heeg1}
    \end{figure}

    Note that $\a_2^\pm$ intersects the fixed point set $C^\pm$ on $\S_\pm$ in exactly two points with opposite sign. 
    Label these two points $x_\text{first}^\pm$ and $x_\text{last}^\pm$, so that when traversing component 2, $x_\text{first}^\pm$ is the very first intersection point between component 2 and the alpha curves, and $x_\text{last}^\pm$ is the last. 
    Additionally, let $p^\pm \in \a_1^\pm \cap C^\pm$ be the intersection point that lies on the handle coming from the generator $\pm x_n$. 
    
    Let 
    \[\bfxpf = \{z_1, \ldots, \t^+(z_k), c_1^1, \ldots, c_{n_1}^1, \ldots, c_1^l, \ldots, c_{n_l}^l\} \in \bbT_\a^+ \cap \widehat{(M^+)^R}\]
    be an intersection point such that $c_1^2 = x_{\text{first}}^+ \in \a_2^+ \cap C^+$. 
    Suppose first that $c_{n_1}^1 = p^+ \in \a_1^+ \cap C^+$. In other words, the point lying on $\a_1^+$ lies on the handle coming from the generator $x_n$. Consider the following three corresponding points:
    \begin{align*}
    \bfxpl &= \{z_1, \ldots, \t^-(z_k), c_1^1, \ldots, c_{n_1}^1, c_2^2, \ldots, c_{n_2}^2,x_{\text{last}}^+,\ldots, c_1^l, \ldots, c_{n_l}^l\} \in \bbT_\a^+ \cap \widehat{(M^+)^R}, \\
    \bfxmf &= \{z_1, \ldots, \t^-(z_k), c_1^1, \ldots, c_{n_1-1}^1, x_{\text{first}}^-, c_2^2, \ldots, c_{n_2}^2, p^-,\ldots, c_1^l, \ldots, c_{n_l}^l\} \in \bbT_\a^- \cap \widehat{(M^-)^R}, \\
    \bfxml &= \{z_1, \ldots, \t^-(z_k), c_1^1, \ldots, c_{n_1-1}^1, c_2^2, \ldots, c_{n_2}^2,p^-,x_{\text{last}}^-,\ldots, c_1^l, \ldots, c_{n_l}^l\} \in \bbT_\a^0 \cap \widehat{(M^0)^R}.
    \end{align*}
    The signs of these intersection points are related as follows: 
    \begin{align*}
    \sign(\bfxpl) &= (-1)(-1)^{n_2-1}\sign(\bfxpf), \\
    \sign(\bfxmf) &= (-1)^{n_2}\sign(\bfxpf), \\
    \sign(\bfxml) &= (-1)(-1)^{n_2}(-1)^{n_2}\sign(\bfxpf).
    \end{align*}
    \mycomment{
    Then $\sign(\bfxmf) = (-1)^{n_2}\sign(\bfxpf)$. We can also consider the points in $\bbT_\a^\pm \cap (\widehat{M^\pm})^R$ we get by replacing $x^\pm_{\text{first}}$ with $x^\pm_{\text{last}}$: 
    \[\bfxpl = \{z_1, \ldots, \t_-(z_k), c_1^1, \ldots, c_{n_1}^1, c_2^2, \ldots, c_{n_2}^2,x_{\text{last}}^+,\ldots, c_1^l, \ldots, c_{n_l}^l\} \in \bbT_\a^+ \cap (\widehat{M^-})^R,\]
    \[\bfxml = \{z_1, \ldots, \t_-(z_k), c_1^1, \ldots, c_{n_1-1}^1, c_2^2, \ldots, c_{n_2}^2,p^-,x_{\text{last}}^-,\ldots, c_1^l, \ldots, c_{n_l}^l\} \in \bbT_\a^+ \cap (\widehat{M^-})^R.\]
    Then $\sign(\bfxpl) = (-1)(-1)^{n_2-1}\sign(\bfxpf)$ and $\sign(\bfxml) = (-1)(-1)^{n_2}(-1)^{n_2}\sign(\bfxpf)$.}
    For this set of four points, there is one corresponding intersection point 
    \[\bfx^0 = \{z_1, \ldots, \t^0(z_k), c_1^1, \ldots, c_{n_1 -1}^1, c_2^2, \ldots, c_{n_2}^2, \ldots, c_1^l, \ldots, c_{n_l}^l\} \in \bbT_\a^0 \cap \widehat{(M^0)^R}\]
    which satisfies $\sign(\bfx^0) = \sign(\bfxpf)$. Therefore,
    \[(\sign(\bfxpf) + \sign(\bfxpl)) - (\sign(\bfxmf) + \sign(\bfxml)) = 2\sign(\bfx^0).\]
    Each intersection point in $\bbT_\a^0 \cap \widehat{(M^0)^R}$ determines a set of four points in $\bbT_\a^\pm \cap \widehat{(M^\pm)^R}$ as above. 
    For a point in $\bbT_\a^+ \cap \widehat{(M^+)^R}$ where the point on $\a_1^+$ does not lie on the handle of $\S_+$ coming from the crossing $x_n$, the corresponding point in $\bbT_\a^- \cap \widehat{(M^-)^R}$ will have the same sign. 
    So, summing over all intersection points, we have that 
    \[\ctot(L_+) - \ctot(L_-) = 2\ctot(L_0).\]

    The other relation is proven similarly. This time, the strands $s$ and $s+1$ involved in the crossing $x_n$ of $L_+$ belong to the same component. 
    Label this component 1 and put a basepoint at the top of strand $s$. 
    Label the rest of the components arbitrarily, fixing a basepoint at the top of each strand where a new component starts. 
    Label the components of $L_-$ and fix basepoints in the same way. 
    Resolving the crossing $x_n$ splits that component into two components. 
    Let component 1 be the the component of $L_0$ containg strand $s$ and let component 2 be the one containing strand $s+1$. 
    Keep the basepoint $w_1$ and fix an additional basepoint $w_2$ at the top of strand $s+1$. 
    The remaining components are then labelled in the same order as that of $L_\pm$, and the same basepoints are kept as well (shifting the label of all of these up by one). 
    Fix alpha curves on $\S_+$ as before, and let $\a_1^+$ be the curve going over the handle coming from $x_n$. 
    There are similar alpha curves on $\S_-$. 
    Finally, the curves $\a_j^+$ for $ j \geq 2$ have analogues $\a_j^0$ on $\S_0$. 
    Instead of $\a_1^+$, take a small simple closed curve $\a_1^0$ on $\S_0$ which encircles $w_2$ and intersects $C^0$ in two points, $x_\text{first}^0$ and $x_\text{last}^0$. 
    As before, using the involution to obtain the beta curves, we have real Heegaard diagrams for $(\S_2(L_+), \t^+, \bfw^+), (\S_2(L_-), \t^-, \bfw^-)$, and $(\S_2(L_0), \t^0, \bfw^0)$. 
    These diagrams are identical away from the handle associated to the crossing being changed, where they appear as in \Cref{fig:skein-reln-heeg2}.
    \begin{figure}[htbp]
    \centering
    \def\svgwidth{.9\textwidth}
    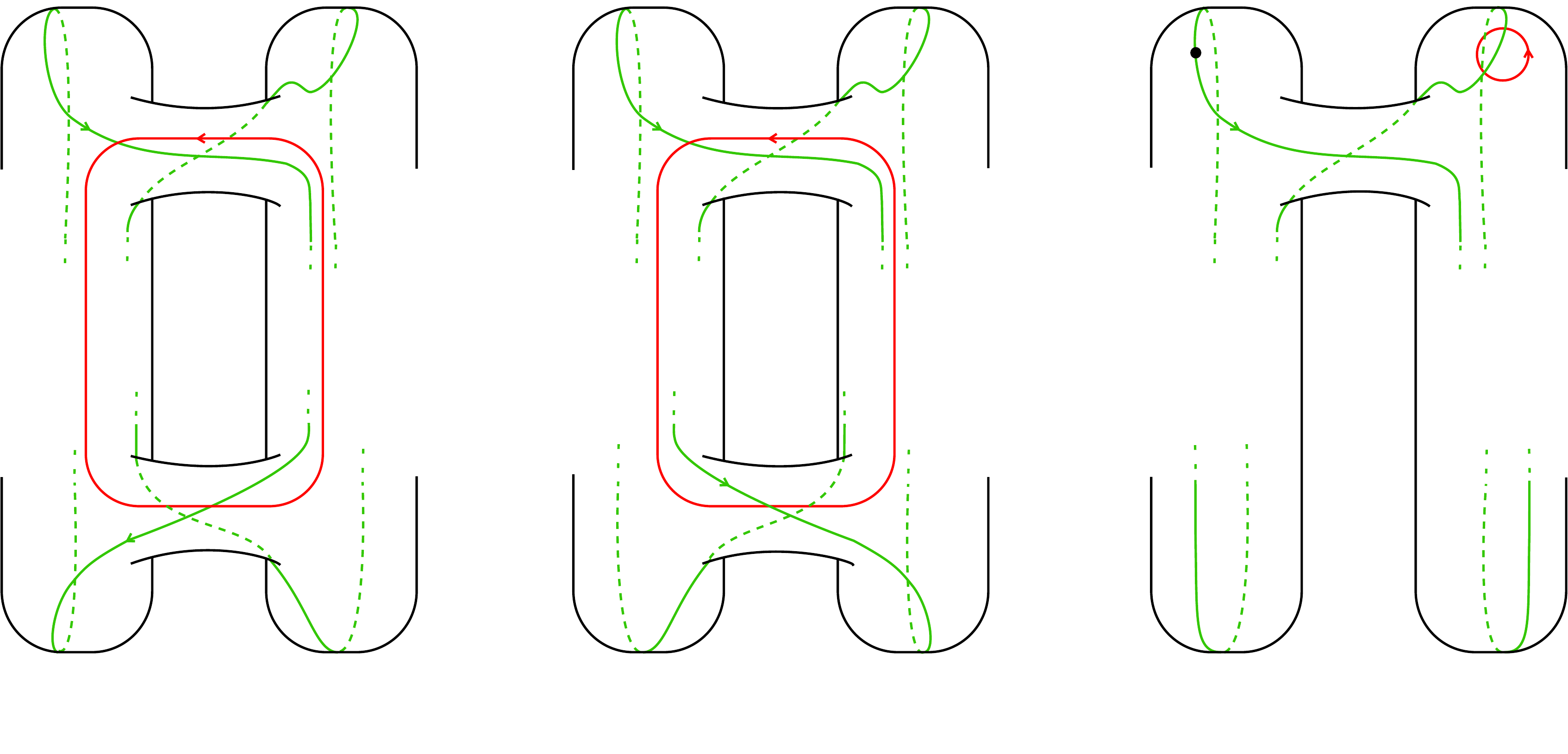
    \caption{Heegaard diagrams for $(\S_2(L_+), \t^+, \bfw^+), (\S_2(L_-), \t^-, \bfw^-)$, and $(\S_2(L_0), \t^0, \bfw^0)$ when $\lvert L_\pm\rvert=\lvert L_0\rvert - 1$.}
    \label{fig:skein-reln-heeg2}
    \end{figure}
    Again, we label the points $p^\pm \in \a_1^\pm \cap C^\pm$.  
    Orient $\a_1^+$ counterclockwise as in the leftmost image of \Cref{fig:skein-reln-heeg2}, and order and orient the rest of the alpha curves so that they determine the canonical orientation on $\bbT_\a^+$. 
    Then we can orient $\a_1^-$ and $\a_1^0$ counterclockwise as in \Cref{fig:skein-reln-heeg2} and order and orient the remaining alpha curves on $\S_-$ and $\S_0$ in the same way as the curves on $\S_+$, so that the resulting sets of alpha curves determine the canonical orientations of $\bbT_\a^-$ and $\bbT_\a^0$. 
    Now we match up intersection points as before. Let 
    \[\bfx^+ = \{z_1, \ldots, \t^+(z_k), c_1^1, \ldots, c_{n_1}^1, \ldots, c_1^l, \ldots, c_{n_l}^l\} \in \bbT_\a^+ \cap \widehat{(M^+)^R}\]
    be an intersection point. 
    As before, we may assume that the point on $\a_1^+$ lies on the handle coming from the crossing $x_n$. 
    Then it must be the case that $c_{n_1}^1  = p^+ \in \a_{1}^+ \cap C^+$. Consider the point
    \[\bfx^- = \{z_1, \ldots, \t^-(z_k), c_1^1, \ldots, c_{i-1}^1, p^-, c_{i}^1, \ldots, c_{n_1-1}^1, \ldots, c_1^l, \ldots, c_{n_l}^l\} \in \bbT_\a^- \cap \widehat{(M^-)^R}.\]
    Here, $p^- \in \a_1^-\cap C^-$ is the intersection point lying on $\a_1^-$, so that $c_1^1, \ldots, c_{i-1}^1$ are the points on component 1 encountered before encountering the crossing $x_n$ for the first time.
    Then we have that $\sign(\bfx^-) = (-1)^{n_1 - i}\sign(\bfx^+)$, so that 
    \[\sign(\bfx^+) - \sign(\bfx^-) = \begin{cases}
        0 & n_1 -i \text{ is even} \\
        2\sign(\bfx^+) & n_1 - i \text{ is odd}
    \end{cases}\]
    Now, consider the following two intersection points in $\bbT_\a^0\cap\widehat{(M^0)^R}$:
    \[\bfx^0_\text{first} = \{z_1, \ldots, \t^0(z_k), c_1^1, \ldots, c_{i-1}^1,x^0_\text{first},c_i^1,\ldots,c_{n_1-1}^1, \ldots, c_1^l, \ldots, c_{n_l}^l\}\]
    \[\bfx^0_\text{last} = \{z_1, \ldots, \t^0(z_k), c_1^1, \ldots, c_{i-1}^1,c_i^1,\ldots,c_{n_1-1}^1, x^0_\text{last},\ldots, c_1^l, \ldots, c_{n_l}^l\}.\]
    We see that $\sign(\bfx_\text{first}^0) = (-1)(-1)^{n_1 - i}\sign(\bfx_\text{last}^0)$. The orientations of $\a_1^+$ and $\a_1^0$ also imply that $\sign(\bfx_\text{last}^0) = -\sign(\bfx^+)$. Therefore, 
    \[
        \sign(\bfx_\text{first}^0) + \sign(\bfx_\text{last}^0) = \begin{cases}
            0 & n_1 - i \text{ is even} \\
            -2\sign(\bfx^+) & n_1 - i \text{ is odd}
        \end{cases}
    \]
    Summing over all intersection points, we have that 
    \[\ctot(L_+) - \ctot(L_-) = -\ctot(L_0).\]   
\end{proof}

In Section 7 of \cite{GM1}, the authors show that $\ctot(L)$ and $\cs(L)$ can be computed algorithmically for any link $L$ in $S^3$. 
Using a computer program in \cite{HFR_python}, the authors present a table of $\ctot(K)$ and $\cs(K)$ for a selection of knots. 
They fix the signs so that the sum of Euler characteristics in each $\text{Spin}^c$-structure is positive. 
By making additions to this program in \cite{HFR_python_links}, we can compute the invariants $\ctot(K)$ and $\cs(K)$ using $\grh$ to determine the signs in our new convention.
We present the table in \Cref{table} below. 

\begin{figure}
\begin{center}
\begin{tabular}{|c|c|c|}
    \hline
     Knot & $\ctot(K)$ & $\cs(K)$ \\
     \hline
     $0_1$ & $1$ & $[1]$ \\
     $3_1$ & $-1$ & $[-1, 1, -1]$ \\
     $4_1$ & $3$ & $[1,1 -1, 1, 1]$ \\
     $5_1$ & $-1$ & $[1, -1, -1, -1, 1]$ \\
     $5_2$ & $-3$ & $[-1, -1, 1, -1, 1, -1, -1]$ \\
     $6_1$ & $5$ & $[1, 1, 1, -1, 1, -1, 1, 1, 1]$ \\
     $6_2$ & $-1$ & $[-1, -1, 1, -1, 1, 1, 1, -1, 1, -1, -1]$\\
     $6_3$ & $3$ & $[1, -1, 1, 1, 1, -1, -1, -1, 1, 1, 1, -1, 1]$ \\
     $7_1$ & $1$ & $[-1, 1, 1, -1, 1, 1, -1]$ \\
     $7_2$ & $-5$ & $[-1, -1, -1, 1, -1, 1, -1, 1, -1, -1, -1]$ \\
     $7_3$ & $-1$ & $[1, 1, -1, 1, -1, -1, -1, -1, -1, 1, -1, 1, 1]$ \\
     $7_4$ & $-7$ & $[-1, -1, 1, -1, 1, -1, -1, -1, -1, -1, 1, -1, 1, -1, -1]$ \\
     $7_5$ & $1$ & $[1, -1, 1, 1, 1, -1, -1, -1, 1, -1, -1, -1, 1, 1, 1, -1, 1]$ \\
     $7_6$ & $-5$ & $[-1, -1, 1, -1, -1, -1, 1, -1, 1, 1, 1, -1, 1, -1, -1, -1, 1, -1, -1]$ \\
     $7_7$ & $7$ & $[1, 1, -1, 1, 1, 1, 1, -1, 1, -1, -1, -1, 1, -1, 1, 1, 1, 1, -1, 1, 1]$ \\
     $8_1$ & $7$ & $[1, 1, -1, 1, 1, 1, -1, 1, 1, 1, -1, 1, 1]$ \\
     $8_2$ & $-3$ & $[1, 1, -1, 1, -1, -1, -1, -1, 1, -1, -1, -1, -1, 1, -1, 1, 1]$ \\
     $8_3$ & $9$ & $[1, 1, 1, -1, 1, -1, 1, 1, 1, 1, 1, -1, 1, -1, 1, 1, 1]$ \\
     $8_4$ & $-1$ & $[-1, -1, -1, 1, -1, 1, -1, 1, 1, 1, 1, 1, -1, 1, -1, 1, -1, -1, -1]$ \\
     $8_5$ & $-1$ & $[1, 1, -1, -1, 1, -1, 1, 1, -1, -1, -1, -1, -1, 1, 1, -1, 1, -1, -1, 1, 1]$ \\
     $8_6$ & $-3$ & $[-1, -1, 1, -1, -1, -1, 1, -1, 1, 1, 1, -1, 1, 1, 1, -1, 1, -1, -1, -1, 1, -1, -1]$ \\
     $8_7$ & $1$ & $[-1, 1, -1, -1, -1, 1, 1, 1, -1, 1, 1, -1, 1, 1, -1, 1, 1, 1, -1, -1, -1, 1, -1]$ \\
     $9_1$ & 1 & $[1, -1, -1, 1, 1, 1, -1, -1, 1]$ \\
     $10_{152}$ & $-1$ & $[1, 1, -1, 1, -1, -3, -1, 1, -1, 1, 1]$ \\
    \hline   
\end{tabular}  
\end{center}
\caption{ }
\label{table}
\end{figure}

There are still a number of questions about this grading.
It is unknown if $\grh$ depends on the choice of primitive homology class $S \in H_2(X, L ; \bbZ)$. 
Although we have seen that the grading on $\widehat{\HFR}$ is independent of the ordering of the link components, it is unclear if the grading on the other versions of $\HFR$ depends on the ordering. 
In the case of double branched covers of links in $S^3$, it would be interesting to interpret the Euler characteristic in each Spin$^c$-structure in terms of known invariants.  

\bibliographystyle{amsalpha}
\bibliography{refs}

\end{document}